\documentclass[a4paper]{amsart}
\usepackage{amsmath,epsf}
\usepackage{mathrsfs}
\usepackage{hyperref}
\usepackage{amssymb,stmaryrd}
\usepackage{enumitem}
\usepackage{tikz} \usetikzlibrary{matrix,arrows,decorations.pathmorphing}

\subjclass[2000]{Primary 03C10, 12H10, 03C60, 11G25. Secondary 14G10, 14G15.}
\keywords{difference scheme, Galois stratification, Galois formula, Frobenius automorphism, ACFA}
\title{Direct twisted Galois stratification}
\date{\today}

\author{Ivan Toma{\v s}i{\'c}}
\address{Ivan Toma{\v s}i{\'c}\\
         School of Mathematical Sciences\\
  	Queen Mary University of London\\
         London, E1 4NS\\
        United Kingdom}
\email{i.tomasic@qmul.ac.uk}
\newtheorem{theorem}{Theorem}[section]
\newtheorem{corollary}[theorem]{Corollary}
\newtheorem{proposition}[theorem]{Proposition}
\newtheorem{lemma}[theorem]{Lemma}
\newtheorem{fact}[theorem]{Fact}

\theoremstyle{definition}
\newtheorem{definition}[theorem]{Definition}
\newtheorem{example}[theorem]{Example}

\theoremstyle{remark}
\newtheorem{remark}[theorem]{Remark}

\newtheorem{notation}[theorem]{Notation}

\providecommand{\Z}{\mathbb{Z}}

\providecommand{\N}{\mathbb{N}}
\providecommand{\p}{\mathfrak{p}}

\providecommand{\cF}{\mathcal{F}}

\providecommand{\cA}{\mathcal{A}}
\providecommand{\cB}{\mathcal{B}}

\providecommand{\kk}{\mathbf{k}}

\providecommand{\spec}{\text{\rm Spec}}
\providecommand{\acfa}{\text{\sc ACFA}}
\providecommand{\frob}{\text{\sc Frob}}
\providecommand{\gen}{\text{\rm gen}}

\providecommand{\Q}{\mathbb{Q}}
\providecommand{\F}{\mathbb{F}}
\providecommand{\Af}{\mathbb{A}}

\providecommand{\id}{\text{\rm id}}
\providecommand{\Gal}{\text{\rm Gal}}

\providecommand{\Sch}{\text{\rm Sch}}
\providecommand{\Set}{\text{\rm Set}}
\providecommand{\Et}{\text{\rm \'Et}}

\providecommand{\Hom}{\text{\rm Hom}}

\providecommand{\Aut}{\text{\rm Aut}}

\newenvironment{quotestd}
               {\list{}{\rightmargin\leftmargin}%
                \item\relax}
               {\endlist}

\providecommand{\ztild}[1]{\rlap{$\smash{\tilde{\phantom{#1}}}$}\rlap{$\mathring{\phantom{#1}\kern1.1ex}$}#1\kern.1ex}

\providecommand{\lexp}[2]{{\vphantom{#2}}^{#1}{\kern-.2ex#2}}

\providecommand{\rlexp}[2]{\lexp{#2}{#1}}

\providecommand{\acirc}[1]{\phantom{a}\llap{$\scriptstyle#1$}
\kern.01ex\lower.75ex\hbox{$\smash{\mathring{}}$}}

\providecommand{\lzexp}[3]{{\vphantom{#2}}^{\lower0.0ex\hbox{\smash{$\acirc{#1}$}}}\kern-.1ex #2}

\providecommand{\lrexp}[3]{{\vphantom{#2}}^{#1}{\kern-.1ex#2^#3}}

\begin{document}
\maketitle
\begin{abstract}
The theory ACFA admits a primitive recursive quantifier elimination procedure. It is therefore primitive recursively decidable.
\end{abstract}

\section{Introduction}
\subsection{Background}
The role of our work in Model Theory of fields with powers of Frobenius and existentially closed difference fields is analogous to the role that Galois stratification of Fried, Haran, Jarden and Sacerdote
(\cite{fried-jarden}, \cite{fried-sacer}, \cite{FHJ}), played in Model Theory of finite and pseudofinite fields, providing a more precise form, as well as the effectivity of quantifier elimination. In this light, our work will have an impact in the study of exceptional difference polynomials, difference version of Davenport's problem, graphs definable in fields with Frobenii and existentially closed difference fields, and many other areas inspired by applications of the classical Galois stratification over finite and pseudofinite fields.

In papers \cite{ive-tgs} and \cite{ive-tgsacfa}, we developed a theory of twisted Galois stratification for generalised difference schemes, and we established a rather \emph{fine quantifier elimination} result, stating that every first-order formula in the language of difference rings is equivalent to a \emph{Galois formula} modulo the theory ACFA of existentially closed difference fields, where the latter formulae are associated with \emph{finite} Galois covers of difference schemes. We argued that the elimination procedure was \emph{effective} in the sense that is was primitive recursive reducible to a few natural operations in difference algebra (the status of which  is unknown at the moment). 

In this paper, we develop \emph{direct twisted Galois stratification} in the context of direct presentations of difference schemes, which approximates the difference scheme framework to a sufficient order. We show a slightly coarser quantifier elimination result, Theorem~\ref{dirqe}, which (informally) states that 
\begin{quotestd}
every first-order formula is equivalent to a \emph{direct Galois formula} modulo ACFA, or over the fields with Frobenii, 
\end{quotestd}
 where the latter formulae are associated with direct Galois covers. 
Even though the class of  direct Galois formulae is coarser than that of Galois formulae, direct Galois formulae are equivalent (\ref{dirgal-fo}) to the $\exists_1$-formulae that appear after the known \emph{logic quantifier elimination} for ACFA from \cite{angus} and \cite{zoe-udi}.  

Our main result (Theorems \ref{effdirqe} and \ref{effdirqe-frob}) is that 
\begin{quotestd}
 the quantifier elimination procedure for ACFA and for fields with Frobenii is %outright 
 \emph{primitive recursive}. 
\end{quotestd}
Given that working with direct presentations essentially reduces to working with algebraic varieties and correspondences between them, this follows by applying methods of classical effective/constructive algebraic geometry in our framework. Consequently, ACFA and the first-order theory of fields with Frobenii are decidable by a primitive recursive procedure, see Corollaries~\ref{prdec} and \ref{prdec-frob}.

The present paper is not a variant of \cite{ive-tgs} and \cite{ive-tgsacfa} since the whole machinery of direct Galois covers had to be developed from first principles, which is significantly more intricate than previous considerations involving difference schemes. There is no direct interaction with the methods of the previous papers, they only provide ideological guidance to identify the main conceptual steppingstones in the stratification procedure.

\subsection{Direct presentations of difference schemes}

Let $(k,\varsigma)$ be a difference field, let $X$ be an algebraic variety over $k$, let $X_\varsigma$ denote the base change of $X$ via 
$\varsigma:k\to k$, and let $W\subseteq X\times X_\varsigma$ be a closed subvariety. Let $(F,\varphi)$ be any difference field extending $(k,\varsigma)$. The intuitive idea that sets of the form
$$
\{x\in X(F):(x,x\varphi)\in W\}
$$
should correspond to sets of $(F,\varphi)$-points of a `difference variety' has been around from the beginning of research on difference algebra, and it was particularly useful in the model-theoretic study of difference fields, as in \cite{angus} and \cite{zoe-udi}.

Although the above data determines a `directly presented difference scheme' defined in \cite{udi}, we choose to minimise the use of the framework of difference schemes, and remain in the context of their \emph{direct presentations}, using the classical language of algebraic schemes and correspondences (we only use difference schemes to control parameters, since the alternative leads to rather cumbersome notation). The benefit of this approach is that we can profit from the methods of effective %(historically called `constructive') 
algebraic geometry in order to prove that our constructions are primitive recursive.

\subsection{Generalised difference schemes}

The logic quantifier elimination mentioned above states that a first-order formula in the language of difference rings is equivalent to an existential formula modulo ACFA. Intuitively, such a formula chooses one of the finitely many possible behaviours of the difference operator on some finite difference ring extension. We encode such an extension through the notion of a \emph{Galois cover} of direct presentations
$$
(X,\Sigma)\to (Y,\sigma),
$$ 
where $\Sigma$ is the set of all possible lifts of the difference operator $\sigma$ from $Y$ to the `finite' cover $X$. The choice of a lift $\tilde{\sigma}\in\Sigma$ `up to isomorphism', i.e., up to the action of the associated Galois group $G$ amounts to the choice of a $G$-conjugacy class $C$ in $\Sigma$. Hence the data 
$$
\langle(X,\Sigma)/(Y,\sigma),C\rangle
$$
constitutes the basic building blocks of our Galois formulae.

In order to afford the existence of Galois covers and to provide a notational device for discussing  several possible lifts of the difference operator, we must make our framework flexible enough to include presentations of \emph{generalised difference schemes}, developed in \cite{ive-tch}. Having that in mind, the real difficulty of this project is to isolate a robust enough
notion of a Galois cover which will  perpetuate through a number of operations and constructions (cf.~Subsection~\ref{ss:constr-dirgal}). 
%(unlike the intrinsic notion of a Galois cover of difference schemes which is naturally preserved by all such constructions). 

\subsection{Organisation of the paper}

In Section~\ref{s:dir-pres} we define (generalised) direct presentations of difference schemes and their morphisms, and consider their structure and basic properties.

In Section~\ref{s:dir-gal-cov}, we define direct Galois covers of direct presentations, consider their basic properties, and show their permanence properties under the constructions needed in the sequel.

In Section~\ref{s:gal-fmla}, we define direct Galois stratifications and their associated formulae, and give a preliminary comparison of Galois formulae to first-order formulae over existentially closed difference fields (and asymptotically over fields with powers of Frobenius).

This work is completed in Section~\ref{s:dirim-th}, where we prove a direct image theorem \ref{cor-dirdirim}, stating that a direct image of a Galois formula by a morphism of direct presentations is equivalent modulo ACFA (or fields with Frobenius) to a Galois formula. Given that existential quantification can be thought of as taking direct images via projections, this immediately implies a quantifier elimination result for the class of Galois formulae, and shows that it coincides with the class of first-order formulae (\ref{dirqe}). We compute direct images along arbitrary morphisms, rather than just projections, in order to benefit from general decompositions of morphisms which allow us to reduce the computation to manageable cases \ref{dirim-finet} and \ref{dirim-geom-conn}/\ref{dirim-geom-conn-frob}.

In Section~\ref{s:eff-qe}, we review the preceding sections with our `effective goggles' on and argue that the quantifier elimination procedure reduces to the known constructions in effective algebraic geometry and is therefore primitive recursive (\ref{effdirqe}). We show how to draw consequences in effective/constructive difference algebra by proving that the \emph{perfect ideal membership problem} is primitive recursive (\ref{pr-perfidmem}).

Finally, Appendix~\ref{s:comp-dir-pres} gives a comparison of the framework of direct presentations and that of directly presented difference schemes.

\section{Directly presented difference schemes}\label{s:dir-pres}

In view of the explanation from the Introduction that we shall be working with direct presentations rather than the associated difference schemes, the title of this section is symbolic and hints at the fact that we borrowed the adjective `direct' from Hrushovski, and that our framework ties in nicely with that of directly presented schemes, as discussed in the Appendix. 

In the sequel, all rings are commutative with identity. By an \emph{algebraic variety} over a ring $R$, we will mean a reduced separated scheme of finite presentation over $\spec(R)$, and morphisms of varieties are assumed to be locally of finite presentation.

\providecommand{\sadir}{\mathscr{D}^\text{\rm a}}
\providecommand{\sdir}{\mathscr{D}}
\providecommand{\vadir}{\mathscr{D}^\text{\rm av}}
\providecommand{\vdir}{\mathscr{D}^\text{\rm v}}

\subsection{Difference rings}

A \emph{difference ring} is a pair 
$$(R,\varsigma)$$ consisting of a ring $R$ and an endomorphism $\varsigma:R\to R$. 

A difference ring homomorphism $$f:(R,\varsigma)\to(R',\varsigma')$$
is a ring homomorphism $f:R\to R'$ satisfying $$\varsigma' f=f\varsigma.$$

\begin{definition}
A difference ring $(R,\varsigma)$ is called:
\begin{enumerate}
\item \emph{inversive}, if $\varsigma$ is bijective;
\item a \emph{transformal domain}, if $R$ is a domain and $\varsigma$ is injective;
\item a \emph{difference field}, if $R$ is a field.
\end{enumerate}
\end{definition}
 
\begin{fact}[{\cite[Chapter 2, Theorem~II]{cohn}}]
Every difference ring $(R,\varsigma)$ with $\varsigma$ injective has an \emph{inversive closure} $(R^\text{\rm inv},\varsigma^\text{\rm inv})$ with the following universal property. There is an embedding $(R,\varsigma)\hookrightarrow(R^\text{\rm inv},\varsigma^\text{\rm inv})$ such that every morphism $(R,\varsigma)\to (S,\tau)$ to an inversive difference ring factors through $R^\text{\rm inv}$.
\end{fact}

\begin{notation}\label{R-minus-n}
Let $(R,\varsigma)$ be a transformal domain. We write 
$$
R_{-n}=(\varsigma^\text{inv})^{-n}(R),
$$ 
considered as a difference subring of $R^{\mathop{\rm inv}}$.
% such that for every $a\in R$, there exists a $b\in R_{-n}$ with $\varsigma^n(b)=a$. 
\end{notation}

\begin{remark}
Let $(R,\varsigma)$ be a difference ring, and let $S\subseteq R$ be a multiplicatively closed subset with $\varsigma(S)\subseteq S$. Then the localisation
$$
S^{-1}R
$$
has a natural structure of a difference ring.
\end{remark}

\begin{definition}
Let $(R,\varsigma)$ be a transformal domain. 
\begin{enumerate}
\item A (finite) $\varsigma$-localisation of $R$ with respect to $f\in R\setminus\{0\}$ is the difference ring $S^{-1}R$, where $S$ is the multiplicative set generated by $\{\varsigma^{i}(f): i\in\N\}$.

\item By convention, whenever we mention a \emph{$\varsigma$-localisation of $R$}, we shall mean a finite $\varsigma$-localisation with respect to some $f$ as above.
\end{enumerate}
\end{definition}

\begin{definition}
The \emph{affine difference scheme} associated to a difference ring $(R,\varsigma)$ is
$$\spec^\varsigma(R)=\{\p\in\spec(R):\varsigma^{-1}(\p)=\p\},$$
together with the \emph{Zariski topology}, as well as the structure sheaf,  induced from $\spec(R)$. It is the fixed point set of 
$$
\spec(R,\varsigma)=(\spec(R),\lexp{a}{\varsigma}).
$$

Write $\mathcal{S}=\spec^\varsigma(R)$.
\begin{enumerate}
\item 
For a point $s\in\mathcal{S}$,  we write $\mathfrak{j}_s$ for the associated $\varsigma$-prime ideal in $R$. 
\item 
The \emph{local ring at $s\in\mathcal{S}$} is just the (difference) localisation $R_{\mathfrak{j}_s}$, and
its residue field $\kk(s)$ is naturally a difference field, equipped with the induced endomorphism $\varsigma^s$.
\item 
A point $s\in\mathcal{S}$ is \emph{closed}, if it is closed in the Zariski topology, i.e., if $\mathfrak{j}_s$ is maximal among the $\varsigma$-prime ideals in $R$. 
\end{enumerate}
\end{definition}

\subsection{Direct presentations}

\begin{notation}
Let $(R,\varsigma)$ be a difference ring, let $S=\spec(R)$ and, by a slight abuse of notation, write $\varsigma$ for the scheme morphism $\lexp{a}{\varsigma}:S\to S$ induced by $\varsigma$. We write
$$
(\mathord{-})_\varsigma:\Sch_{\mathord{/}S}\to \Sch_{\mathord{/}S}
$$  
for the \emph{base change functor via $\varsigma$}. For an $S$-scheme $Y$, its $\varsigma$-twist is
$$
Y_\varsigma=Y\times_SS
$$
as shown in the cartesian diagram
$$
 \begin{tikzpicture} 
 [cross line/.style={preaction={draw=white, -,
line width=3pt}}]
\matrix(m)[matrix of math nodes, minimum size=1.7em,
inner sep=0pt,
row sep=1.5em, column sep=1em, text height=1.5ex, text depth=0.25ex]
 { 
                        |(P)|{Y_{\varsigma}} 	& |(2)| {Y}          \\
                       |(1)|{S}             & |(h)|{S} \\};
\path[->,font=\scriptsize,>=to, thin]
(P) edge  (1) edge (2)
(1) edge node[above]{$\varsigma$}  (h)
(2) edge  (h)
;
\end{tikzpicture}
$$
and we extend the definition to morphisms in a natural way.
\end{notation}

\begin{definition}
Let $(R,\varsigma)$ be a transformal domain, and write $S=\spec(R)$.
We define the category of \emph{almost direct presentations} $\sadir=\sadir_{(R,\varsigma)}$ as follows.
\begin{enumerate}
\item\label{dpdiag} An object $(X,\Sigma)$ consists of $S$-schemes $X_0$ and $X_1$ and a collection of commutative diagrams of scheme morphisms
\begin{center}
 \begin{tikzpicture} 
\matrix(m)[matrix of math nodes, row sep=2em, column sep=2em, text height=1.5ex, text depth=0.25ex]
 {
 |(2)|{X_0}		& |(3)|{X_1} 	& |(4)|{X_0}\\
 |(l2)|{S}		& |(l3)|{S} 		& |(l4)|{S}\\
 }; 
\path[->,font=\scriptsize,>=to, thin]
(3) edge node[above]{${\pi_1}$} (2) (3) edge node[above]{${\sigma}$}   (4)
(l3) edge node[above]{${\mathop{\rm id}}$} (l2) (l3) edge node[above]{${\varsigma}$} (l4)
(2) edge  (l2) (3) edge (l3) (4) edge  (l4);
\end{tikzpicture}
\end{center}
indexed by $\sigma\in\Sigma$. Note that $\pi_1$ is a morphism of $S$-schemes, while each $\sigma$ is only a $\varsigma$-linear scheme morphism. 

\item A morphism $f:(X,\Sigma)\to(Y,T)$ consists of a map $\rlexp{()}{f}:\Sigma\to T$ and $S$-morphisms $f_0:X_0\to Y_0$, $f_1:X_1\to Y_1$, making the diagram
$$
 \begin{tikzpicture}
[cross line/.style={preaction={draw=white, -,
line width=4pt}}]
\matrix(m)[matrix of math nodes, row sep=.9em, column sep=.5em, text height=1.5ex, text depth=0.25ex]
{			& |(x0)| {X_0}	&			& |(x1)| {X_1} 	&			& |(x0s)| {X_0}	\\   [.2em]
|(y0)|{Y_0} 	&			& |(y1)|{Y_1} 	&			&  |(y0s)| {Y_0}	&			\\  [.4em]
%		& |(d4)| {X}	&			& |(d3)| {Y}	\\   %[.8cm]
			& |(s0)|{S} 		&			& |(s1)|{S} 	&			& |(s0s)|{S} 				\\};
\path[->,font=\scriptsize,>=to, thin]
(x1) edge node[above]{$\pi_1$} (x0) edge node[above]{$\sigma$} (x0s) 
	edge node[left,pos=0.3]{$f_1$} (y1) edge (s1)
(x0) edge node[left,pos=0.3]{$f_0$} (y0) edge (s0)
(x0s) edge node[left,pos=0.3]{$f_0$} (y0s) edge (s0s)
(y0) edge (s0)
(y0s) edge (s0s)
(y1) edge [cross line] node[above,pos=0.2]{$\pi_1$} (y0) edge [cross line] node[above,pos=0.8]{$\rlexp{\sigma}{f}$} (y0s)  edge (s1)
(s1) edge node[above]{$\mathop{\rm id}$} (s0) edge node[above]{$\varsigma$} (s0s) 
%(d4) edge node[pos=0.25,left]{$\lexp{a}{\sigma}$} (u4)  edge node[pos=0.25, below]{$\lexp{a}{\varphi}$} (d3)

%(u1) edge [cross line] node[pos=.8, above]{$\lzexp{a}{\varphi}{0}$} (u2)  edge node[auto]{$i$} (u4)
	
%(d2) edge [cross line] node[pos=0.75, right]{$\lzexp{a}{\tau}{0}$} (u2)  edge node[below]{$j$} (d3)
%(u4) edge node[above]{$\lexp{a}{\varphi}$} (u3)	
%(u2) edge node[pos=0.4,above]{$j$} (u3)	
%(d3) edge node[right]{$\lexp{a}{\tau}$} (u3);
;
\end{tikzpicture}
$$ 
commutative for every $\sigma\in\Sigma$.
\end{enumerate}
The category of \emph{direct presentations} is the full subcategory $\sdir$ of $\sadir$ consisting of those objects $(X,\Sigma)$ for which the diagram in
(\ref{dpdiag}) induces a closed immersion $X_1\hookrightarrow X_0\times_SX_0$ %$X_1\hookrightarrow X_0\times_SX_{0,\varsigma}$ 
for every $\sigma\in\Sigma$.
\end{definition}

\begin{definition}
Let $(R,\varsigma)$ be a transformal domain, and write $S=\spec(R)$.
We define the category $\vadir=\vadir_{(R,\varsigma)}$ as follows.
\begin{enumerate}
\item\label{dvpdiag} An object $(X,\Sigma)$ consists of $S$-schemes $X_0$ and $X_1$ and a diagram of $S$-morphisms
$$
 \begin{tikzpicture} 
\matrix(m)[matrix of math nodes, row sep=2em, column sep=2em, text height=1.5ex, text depth=0.25ex]
 {
 |(2)|{X_0}		& |(3)|{X_1} 	& |(4)|{X_{0\varsigma}}\\
 }; 
\path[->,font=\scriptsize,>=to, thin]
(3) edge node[above]{${\pi_1}$} (2) (3) edge node[above]{${\pi_2(\sigma)}$}   (4);
%(l3) edge node[above]{${\mathop{\rm id}}$} (l2) (l3) edge node[above]{${\varsigma}$} (l4)
%(2) edge  (l2) (3) edge (l3) (4) edge  (l4);
\end{tikzpicture}
$$
for every $\sigma\in\Sigma$.
\item A morphism $f:(X,\Sigma)\to(Y,T)$ consists of a map $\rlexp{()}{f}:\Sigma\to T$ and $S$-morphisms $f_0:X_0\to Y_0$, $f_1:X_1\to Y_1$, making the diagram
$$
 \begin{tikzpicture}
[cross line/.style={preaction={draw=white, -,
line width=4pt}}]
\matrix(m)[matrix of math nodes, row sep=2.2em, column sep=2.2em, text height=1.5ex, text depth=0.25ex]
{
|(x0)| {X_0}		&			 |(x1)| {X_1} 			& |(x0s)| {X_{0\varsigma}}	\\  
|(y0)|{Y_0} 		&			 |(y1)|{Y_1} 			& |(y0s)|{Y_{0\varsigma}} 				\\};
\path[->,font=\scriptsize,>=to, thin]
(x1) edge node[above]{$\pi_1$} (x0) edge node[above]{$\pi_2(\sigma)$} (x0s) 
	edge node[left,pos=0.5]{$f_1$} (y1) 
(x0) edge node[left,pos=0.5]{$f_0$} (y0)
(x0s) edge node[right,pos=0.5]{$f_{0\varsigma}$} (y0s)
(y1) edge node[above]{$\pi_1$} (y0) edge node[above]{$\pi_2(\rlexp{\sigma}{f})$} (y0s) 
;
\end{tikzpicture}
$$ 
commutative for every $\sigma\in\Sigma$.
\end{enumerate}
The category $\vdir$ is the full subcategory of $\vadir$ consisting of those objects $(X,\Sigma)$ for which the diagram in
(\ref{dvpdiag}) induces a closed immersion $X_1\hookrightarrow X_0\times_SX_{0,\varsigma}$ 
for every $\sigma\in\Sigma$.
\end{definition}

\begin{remark}
The commutative diagram
$$
 \begin{tikzpicture} 
 [cross line/.style={preaction={draw=white, -,
line width=3pt}}]
\matrix(m)[matrix of math nodes, minimum size=1.7em,
inner sep=0pt,
row sep=1.5em, column sep=1em, text height=1.5ex, text depth=0.25ex]
 { 
 	 |(3)|{X_1}  & 			&[2em]			\\
                       & |(P)|{X_{0\varsigma}} 	& |(2)| {X_0}          \\[1em]
                       &|(1)|{S}             & |(h)|{S} \\};
\path[->,font=\scriptsize,>=to, thin]
(P) edge  (1) edge (2)
(1) edge node[above]{$\varsigma$}  (h)
(2) edge[cross line]  (h)
(3) edge [bend right=10] (1) edge [bend left=10] node[above]{$\sigma$} (2) edge[dashed] node[pos=0.75,above right=-4pt]{$\pi_2(\sigma)$} (P)
;
\end{tikzpicture}
$$
yields an equivalence of categories between $\sadir$ and $\vadir$, as well as $\sdir$ and $\vdir$. While the reader used to working with varieties might prefer to work in $\vadir$, certain formalisms become tedious in that context. We therefore work with $\sadir$ and $\vadir$ interchangeably, bearing in mind that the type of an object automatically reveals the category it belongs to.
\end{remark}

\subsection{Points, realisations, fibres}

\begin{notation}
Let $(R,\varsigma)$ be a difference ring, and let $(F,\varphi)$ be an $(R,\varsigma)$-algebra (furnished with a morphism $(R,\varsigma)\to(F,\varphi)$). When needed, we consider 
$\spec(F,\varphi)$ as the object
$$\spec(F)\stackrel{\mathop{\rm id}}{\leftarrow}\spec(F)\stackrel{\varphi}{\rightarrow}\spec(F)$$ of $\sadir$. 
\end{notation}

\begin{definition}\label{dirpts}
Let $(X,\Sigma)$ be an object of $\sadir_{(R,\varsigma)}$ and let $(F,\varphi)$ be an $(R,\varsigma)$-algebra.
\begin{enumerate}
\item The set of \emph{$(F,\varphi)$-points} of $(X,\Sigma)$ is
\begin{equation*}
\begin{split}
(X,\Sigma)(F,\varphi) & =\Hom_{\sadir}(\spec(F,\varphi),(X,\Sigma)) \\
& \simeq \{x_1\in X_1(F): \sigma x_1=\pi_1x_1\varphi\text{\rm, for some }\sigma\in\Sigma\},
\end{split}
\end{equation*}

\item The set of \emph{$(F,\varphi)$-realisations} of an object $(X,\Sigma)$ is
\begin{equation*}
\begin{split}
(X,\Sigma)^\sharp(F,\varphi) &=\{x_0\in X_0(F):x\in(X,\Sigma)(F,\varphi) \}\\
&\simeq \pi_1\{x_1: x_1\in X_1(F), \ \sigma x_1=\pi_1x_1\varphi\text{\rm, for some }\sigma\in\Sigma\}.
\end{split}
\end{equation*}
\end{enumerate}
\end{definition}

\begin{remark}
The category $\sadir_{(R,\varsigma)}$ has products. Indeed, if $(X,\Sigma)$ and $(X',\Sigma')$ are objects of $\sadir_{(R,\varsigma)}$, then, writing $S=\spec(R)$,
$$
(X_0\times_SX_0',X_1\times_SX_1', \Sigma\times\Sigma')
$$
is their product in $\sadir_{(R,\varsigma)}$.
\end{remark}

\begin{definition}
Let $(X,\Sigma)$ be an object of $\sadir_{(R,\varsigma)}$ and let $s\in\spec^\varsigma(R)$. Then we consider $\spec(\kk(s),\varsigma^s)$ as an object in $\sadir_{(R,\varsigma)}$,
and we define the \emph{fibre} $(X_s,\Sigma_s)$ as the product
$$
(X,\Sigma)\times_{\spec(R,\varsigma)}\spec(\kk(s),\varsigma^s),$$
considered as an object of $\sadir_{(\kk(s),\varsigma^s)}$.

\end{definition}

\begin{remark}\label{rem-R-alg}
It is often beneficial to view an object $(X,\Sigma)$ of $\sadir_{(R,\varsigma)}$ as a family of objects $(X_s,\Sigma_s)$ parametrised by points $s\in\mathcal{S}=\spec^\varsigma(R)$.

If $x\in(X,\Sigma)(F,\varphi)$ as in \ref{dirpts} with $F$ a field, then it implicitly determines a homomorphism $(R,\varsigma)\to(F,\varphi)$, whose kernel is a $\varsigma$-prime ideal corresponding to some $s\in\mathcal{S}$ and $(F,\varphi)$ extends $(\kk(s),\varphi^s)$, so in fact we could write $x\in X_s(F,\varphi)$. Later on, when we become more mindful about the role of parameters, we may choose a parameter $s\in\mathcal{S}$ first, and then a field $(F,\varphi)$ extending $(\kk(s),\varsigma^s)$.

\end{remark}

\begin{remark}
For a direct presentation $(X,\sigma)$, %in view of \ref{dir-diff-corr} and \ref{dirpts}, 
we have a bijection
%$$[\varsigma]_R(X_0,X_1)(F,\varphi)=
$$(X,\sigma)^\sharp(F,\varphi)\simeq(X,\sigma)(F,\varphi).$$
\end{remark}
Intuitively speaking, the set of $(F,\varphi)$-points of a `directly presented difference scheme' associated with a direct presentation $(X,\sigma)$ coincides with those in the above Remark (the precise statement is \ref{dirprespts}). This justifies somewhat our habit to refer to the objects of $\sadir$ as `(almost) directly presented difference schemes'.

\subsection{Functors of points and realisations}

\begin{remark}\label{funct-pt-real}
Let $(X,\Sigma)$ be an object of $\sadir_{(R,\varsigma)}$.
Items (1) and (2) from \ref{dirpts} define the functors 
$$(X,\Sigma)^\flat\ \ \ \text{ and }\ \ \ (X,\Sigma)^\sharp$$
 from the category of $(R,\varsigma)$-algebras to the category of sets. We shall write $X$ in place of $X^\flat$ whenever it is clear from the context that we wish to refer to the functor of points.
\end{remark}

We extend the above notation to a context suitable for our intended arithmetical applications.

\begin{definition}\label{funct-pt-real-res}
Let $(R,\varsigma)$ be a transformal domain. % and let $\mathcal{S}=\spec^\varsigma(R)$. 
We consider a category consisting of pairs 
$$
(s,(F,\varphi)),
$$
where $s$ ranges in a subset of $\spec^\varsigma(R)$ and $(F,\varphi)$ belongs to a chosen class of difference fields
% is an
%algebraically closed difference field 
extending $(\kk(s),\sigma^s)$. A morphism between
$(s,(F,\varphi))$ and $(s',(F',\varphi'))$ exists only when $s$ is a specialisation of $s'$, and it is then given by a diagram
$$
 \begin{tikzpicture} 
 [cross line/.style={preaction={draw=white, -,
line width=3pt}}]
\matrix(m)[matrix of math nodes, minimum size=1.7em,
inner sep=0pt,
row sep=1.5em, column sep=1.2em, text height=1.5ex, text depth=0.25ex]
 { 
                        |(P)|{F} 	& |(2)| {F'}          \\
                       |(1)|{\kk(s)}             & |(h)|{\kk(s')} \\};
\path[->,font=\scriptsize,>=to, thin]
(P) edge  (2)
(1) edge  (h)
(1) edge (P)
(h) edge  (2)
;
\end{tikzpicture}
$$
of difference field extensions.

The \emph{points functor} $X^\flat$ and the \emph{realisation functor} $X^\sharp$ associated to an object $(X,\sigma)$ of $\sadir_{(R,\varsigma)}$ are set-valued functors on the above category defined by
\begin{enumerate}
\item
$X^\flat(s,(F,\varphi))=X_s(F,\varphi);$     
\item 
$X^\sharp(s,(F,\varphi))=X_s^\sharp(F,\varphi),$
\end{enumerate}
so that  we have the relation $$X_s^\sharp(F,\varphi)=\pi_{1,s}\{x_1:x\in X_s(F,\varphi)\}=\{x_0:x\in X_s(F,\varphi)\}.$$
An \emph{$(R,\varsigma)$-subassignment} of $X^\sharp$ is any subfunctor $\cF$ of $X^\sharp$.
Namely, for any $(s,(F,\varphi))$ as above,
$$
\cF(s,(F,\varphi))\subseteq X_s^\sharp(F,\varphi),
$$
and for any $u:(s,(F,\varphi))\to(s',(F',\varphi'))$, 
$\cF(u)$ is the restriction of $X^\sharp(u)$ to $\cF(s,(F,\varphi))$. Similarly we define subassignments of $X^\flat$.
\end{definition}

%\begin{notation}
%In order to simplify notation, we shall henceforth write $X$ for the points functor $X^\flat$ whenever its use can be inferred from the context.
%\end{notation}

Note that, in view of \ref{rem-R-alg}, the functors from \ref{funct-pt-real-res} are just restrictions of those from \ref{funct-pt-real} to appropriate subcategories of the category of $(R,\varsigma)$-algebras.

\subsection{Properties of directly presented schemes}

\begin{definition}
Let $f:(X,\sigma)\to(Y,\sigma)$ be a morphism in $\vadir_{(R,\varsigma)}$, let $P$ be a property of $R$-schemes, and let $P'$ be a property of morphisms of $R$-schemes. We say that $X$ is \emph{directly $P$}, if $X_0$, $X_1$ and $X_{0\varsigma}$ have the property $P$. Similarly, we say that
$f$ is \emph{directly $P'$}, if the morphisms $f_0:X_0\to Y_0$, $f_1:X_1\to Y_1$ and $f_{0\varsigma}:X_{0\varsigma}\to Y_{0\varsigma}$ all have property $P'$. 
\end{definition}
In particular, an object $(X,\sigma)$ od $\vadir_{(R,\varsigma)}$ is called a \emph{direct variety} if $X_0$ and $X_1$ are $R$-varieties. 
\begin{definition}
A direct variety $(X,\sigma)$ %in $\vadir$ 
is \emph{H-direct}, if it is directly integral and the projections $\pi_1:X_1\to X_0$ and $\pi_2(\sigma):X_1\to X_{0\varsigma}$ are both dominant.
\end{definition}

\subsection{Decomposition into direct components}

The following `direct decomposition' algorithm is so natural that variants of it already appeared in numerous sources, for example as \cite[Solution to Problem I*, Chapter 8, no.~14]{cohn}, \cite[Proposition~2.2.1]{udi-mm} and \cite[Lemma~3.6]{ive-mark}.

\begin{proposition}\label{direct-decomp}
Let  $(X,\sigma)$ be a direct variety in $\vadir_{(R,\varsigma)}$, where $(R,\varsigma)$ is a transformal domain. Let $\eta$ be the generic point of $\spec^\varsigma(R)$ and write $X_{\eta}$ for the generic fibre of $X$ over $R$. 
We can find a finite %(bounded???) 
number of $H$-direct directly closed subschemes $(Y_i,\sigma)$ defined over a $\varsigma$-localisation $R'$ of $R_{-\dim(X_{0,\eta})}$ such that, for every difference field $(F,\varphi)$ over $(R',\varsigma)$, %with a morphism $(R_{{-}\dim(X_0)},\varsigma)\to(F,\varphi)$,
$$
X^\sharp(F,\varphi)=\cup_iY_i^\sharp(F,\varphi).
$$
\end{proposition}

\begin{proof}
Suppose $(X,\sigma)$ is given by a correspondence 
$X_0\stackrel{\pi_1}{\leftarrow}W\stackrel{\pi_2}{\rightarrow}X_{0\varsigma}$. By decomposing $W_\eta$ into irreducible components and $\varsigma$-localising $R$, we may assume $W$ is irreducible. Let $X_1$ be the Zariski closure of $\pi_1(W)$ in $X_0$, %(scheme-th image??), 
and let $X_2$ be the Zariski closure of $\pi_2(W)$ in $X_{0\varsigma}$. It follows that $X_1$ and $X_2$ are irreducible. If $X_{1\varsigma}=X_2$ the construction ends with the H-direct $X_1\stackrel{\pi_1}{\leftarrow}W\stackrel{\pi_2}{\rightarrow}X_{1\varsigma}$.

Otherwise, we consider the presentation $(X',\sigma)$ determined by $X'_0\leftarrow W'\rightarrow X'_{0\varsigma}$, where
$X'_0=X_1\cap X_{2\varsigma^{-1}}$ and $W'=(X'_0\times X'_{0\varsigma})\times_{X_0\times X_{0\varsigma}} W$ are both defined over $R_{{-}1}$. It is straightforward to verify that $(X,\sigma)^\sharp(F,\varphi)=(X',\sigma)^\sharp(F,\varphi)$ for any suitable $(F,\varphi)$. 
Moreover, denoting by $\eta'$ the generic point of $\spec^\varsigma(R_{-1})$, since $\dim(X'_{0,\eta'})<\dim(X_{0,\eta'})=\dim(X_{0,\eta})$, 
%$\dim(W')\leq\dim(W)$, 
we can continue by induction on dimension which clearly ends in at most $\dim(X_{0,\eta})$ steps.
\end{proof}

\subsection{Local properties of directly presented schemes}

In this subsection we work over a transformal domain $(R,\varsigma)$, and we implicitly allow a $\varsigma$-localisation of $R$ in every step that requires it.

\begin{proposition}\label{loc-improve}
Let $f:(X,\sigma)\to(Y,\sigma)$ be a morphism of direct varieties in $\sadir$, and let $P$ be a local property of morphisms of algebraic schemes (varieties) which is stable under base change. 
\begin{enumerate}
\item\label{dirgentgt} If $P$ is generic in the target, then the property of being directly $P$ is directly generic in the target.
\item\label{dirgensce} If $P$ is generic in the source, then the property of being directly $P$ is directly generic in the source.  
\end{enumerate}
\end{proposition}
\begin{proof}
For (\ref{dirgentgt}), by genericity in the target, let $V_i\subseteq Y_i$ be open such that $f_i\restriction_{f_i^{-1}(V_i)}$ has property $P$, for $i=0,1$. By base change, $V_{0\varsigma}$ works for $f_{0\varsigma}$. Let $V=\pi_1^{-1}(V_0)\cap\pi_2^{-1}(V_{0\varsigma})\cap V_1$. Then 
$$
\begin{tikzpicture} 
\matrix(m)[matrix of math nodes, row sep=2em, column sep=1em, text height=1.5ex, text depth=0.25ex]
 {
 |(2)|{f_0^{-1}(V_0)}		& |(3)|{f_1^{-1}(V)} 	& |(4)|{f_{0\varsigma}^{-1}(V_{0\varsigma})}\\
 |(l2)|{V_0}		& |(l3)|{V} 		& |(l4)|{V_{0\varsigma}}\\
 }; 
\path[->,font=\scriptsize,>=to, thin]
(3) edge node[above]{} (2) (3) edge node[above]{}   (4)
(l3) edge node[above]{} (l2) (l3) edge node[above]{} (l4)
(2) edge  (l2) (3) edge (l3) (4) edge  (l4);
\end{tikzpicture}
$$
is directly $P$.

In the case (\ref{dirgensce}) of genericity in the source, let $V_i\subseteq Y_i$, $U_i\subseteq X_i$ be open such that
$f_i\restriction_{U_i\cap f_i^{-1}(V_i)}$ has property $P$, for $i=0,1$. By base change, $V_{0\varsigma}$ and $U_{0\varsigma}$ work for $f_{0\varsigma}$. Let $V=(\pi_1^Y)^{-1}(V_0)\cap(\pi_2^Y)^{-1}(V_{0\varsigma})\cap V_1$ and $U=(\pi_1^X)^{-1}(U_0)\cap(\pi_2^X)^{-1}(U_{0\varsigma})\cap U_1$. Then 
$$
\begin{tikzpicture} 
\matrix(m)[matrix of math nodes, row sep=2em, column sep=1em, text height=1.5ex, text depth=0.25ex]
 {
 |(2)|{U_0\cap f_0^{-1}(V_0)}		& |(3)|{U\cap f_1^{-1}(V)} 	& |(4)|{U_{0\varsigma}\cap f_{0\varsigma}^{-1}(V_{0\varsigma})}\\
 |(l2)|{V_0}		& |(l3)|{V} 		& |(l4)|{V_{0\varsigma}}\\
 }; 
\path[->,font=\scriptsize,>=to, thin]
(3) edge node[above]{} (2) (3) edge node[above]{}   (4)
(l3) edge node[above]{} (l2) (l3) edge node[above]{} (l4)
(2) edge  (l2) (3) edge (l3) (4) edge  (l4);
\end{tikzpicture}
$$
is directly $P$.
\end{proof}

\begin{corollary}\label{loc-improvement}
Let $f:(X,\sigma)\to(Y,\sigma)$ be a morphism of direct varieties in $\sadir_{(R,\varsigma)}$.
\begin{enumerate}
\item If $f$ is a map of directly integral schemes which has directly generically integral fibres, there is a direct localisation of $(Y,\sigma)$ over which $f$ is directly universally submersive (cf.~\ref{def-submer}, \ref{loc-submer}) with geometrically integral fibres.
\item If $f$ is directly generically \'etale,  there is a direct localisation of $(Y,\sigma)$ over which $f$ is directly finite \'etale.
\item If $f$ is directly generically smooth, there is a direct localisation $X'$ of $X$ and $Y'$ of $Y$ such that $f\restriction_{X'\cap f^{-1}(Y')}$ is directly smooth.  
\item If $(X,\sigma)$ is directly generically smooth (over $(R,\varsigma)$), there is a direct localisation of $X$ which is directly normal.
\end{enumerate}
\end{corollary}

\begin{proposition}
Let $f:(X,\sigma)\to(Y,\sigma)$ be a morphism of direct varieties in $\sadir_{(R,\varsigma)}$ and let $P$ be a property of morphisms of algebraic schemes which is generic in the source (or target). 
There exists stratifications of $(X,\sigma)$ and $(Y,\sigma)$ into finitely many directly integral locally closed sub-objects $(X_i,\sigma)$, $(Y_j,\sigma)$ (defined over a $\varsigma$-localisation of some $R_{-n}$) such that each restriction $f_i:(X_i,\sigma)\to(Y_{f(i)},\sigma)$ of $f$ is directly $P$.
\end{proposition}
%!!!!!!!!!!!!!  NB these are going to be on a $\varsigma$-localisation of some $R_{-n}$.
\begin{proof}
By \ref{direct-decomp} we may assume that $X$ and $Y$ are directly integral, and
by \ref{loc-improve}, we find localisations $(U,\sigma)$ of $X$ and $(V,\sigma)$ of $Y$ so that $f\restriction_U:U \to V$ is directly $P$. In the remaining complement
$$
\begin{tikzpicture} 
\matrix(m)[matrix of math nodes, row sep=2em, column sep=1em, text height=1.5ex, text depth=0.25ex]
 {
 |(2)|{X_0}		& |(3)|{X_1\setminus U_1} 	& |(4)|{X_{0\varsigma}}\\
 |(l2)|{Y_0}		& |(l3)|{\overline{f_1(X_1\setminus U_1)}} 		& |(l4)|{Y_{0\varsigma}}\\
 }; 
\path[->,font=\scriptsize,>=to, thin]
(3) edge node[above]{} (2) (3) edge node[above]{}   (4)
(l3) edge node[above]{} (l2) (l3) edge node[above]{} (l4)
(2) edge  (l2) (3) edge (l3) (4) edge  (l4);
\end{tikzpicture}
$$
the dimension of $X_1\setminus U_1$ is strictly lower than the dimension of $X_1$, and we continue by devissage.
\end{proof}

\section{Direct Galois covers}\label{s:dir-gal-cov}

\subsection{Classical Galois covers}

%\begin{definition}\label{def-dir-gal}
%An \emph{almost direct Galois cover} is a $\sadir$-morphism  $(X,\Sigma)\to(Y,\sigma)$ such that:
%\begin{enumerate}
%\item $X_i/Y_i$ is an algebraic (finite \'etale) Galois cover with group $G_i$, i=0,1;
%\item $\Sigma=G_0\tilde{\sigma}$, for some $\tilde{\sigma}\in\Sigma$;
%%\item $\aut_{\sadir}((X,\Sigma)/(Y,\sigma))\simeq G_1$;
%\item there exists a homomorphism $()^{\pi_1}:G_1\to G_0$ such that $\pi_1 g_1=g_1^{\pi_1}\pi_1$, for $g_1\in G_1$;
%\item\label{zvizda} for every $\tau\in \Sigma$, there exists a homomorphism $()^\tau:G_1\to G_0$ such that $\tau g_1=g_1^\tau\tau$ for $g_1\in G_1$.
%\end{enumerate}
%In this case we will say that $(G,\tilde{\Sigma})$ is the \emph{(almost) direct Galois group} of $(X,\Sigma)/(Y,\sigma)$, where $\tilde{\Sigma}=\{()^\tau:\tau\in\Sigma\}$. A \emph{direct Galois cover} is an almost direct Galois cover in $\sdir$.
%\end{definition}

We recall Grothendieck's theory of the \'etale fundamental group and extract  
some of the basic properties of Galois covers.

All schemes in this subsection are assumed to be  locally noetherian. 
A  finite \'etale morphism $X\to Y$ is called an \emph{\'etale cover}.   If $X$ and $X'$ are two \'etale covers of $Y$, we say that $X$ \emph{dominates} $X'$, if there exists a $Y$-morphism $X\to X'$.

\begin{definition}
Let $S$ be a scheme, and let $\bar{s}\in S(\Omega)$ be a geometric point of $S$ (where $\Omega$ is an algebraically closed field). Let $F=F_{\bar{s}}:\Et(S)\to\Set$ be the \emph{fibre functor} from the category of \'etale covers of $S$ to the category of sets given by
$$
F_{\bar{s}}(X)=X_s=X\times_S\spec(\Omega).
$$
The \emph{\'etale fundamental group} of $S$ (with base point $\bar{s}$) is defined as the profinite group
$$
\pi_1(S,\bar{s})=\Aut(F_{\bar{s}}).
$$ 
\end{definition}
\begin{fact}[\cite{sga1}]\label{groth-gal}
With the above notation, the fibre functor $F_{\bar{s}}$ defines an equivalence of categories between $\Et(S)$ and the category of $\pi_1(S,\bar{s})$-sets.

If $f:S\to S'$ is a morphism, the base change along $f$ functor $\Et(S')\to\Et(S)$ gives rise to a homomorphism
$$
\pi_1(f):\pi_1(S,\bar{s})\to\pi_1(S',f(\bar{s})).
$$
\end{fact}

\begin{definition}
Let $X\to Y$ be a connected \'etale cover. %with $X$ connected. 
Let $\bar{y}$ be a geometric point of $Y$, and let $F_{\bar{y}}(X)=X_{\bar{y}}$ be the geometric fibre of $X$ over $\bar{y}$. We say that $X$ is a \emph{Galois cover} of $Y$ if $\Aut(X/Y)$ acts simply transitively on $F_{\bar{y}}(X)$.
\end{definition}

\begin{fact}
\begin{enumerate}
\item In the correspondence of \ref{groth-gal}, $X\to Y$ is a Galois cover if and only if $F_{\bar{y}}(X)$ is a transitive $\pi_1(Y,\bar{y})$-set.
\item If $X$ is a Galois cover of $Y$ with group $G$, then $Y\simeq X/G$. 
\end{enumerate}
\end{fact}

\begin{fact}
If $X\to Y$ is a connected  \'etale cover, then there exists a least Galois cover $\tilde{X}\to Y$ which dominates $X$ as in the diagram
$$
\begin{tikzpicture} 
\matrix(m)[matrix of math nodes, row sep=1em, column sep=1em,text height=1.3ex, text depth=0.25ex]
{       
 		& |(u2)|{\tilde{X}}	\\
 |(d1)|{X} 					& 								\\
 								& |(b)|{Y}            					\\}; 
\path[->,font=\scriptsize,>=to, thin]
(u2) edge %[dashed]  
(d1)
(d1) edge (b)
(u2) edge %[dashed] 
(b);
\end{tikzpicture}
$$
in the sense that any other Galois cover $Z$ of $Y$ that dominates $X$ also dominates $\tilde{X}$. Such an $\tilde{X}$ is unique up to isomorphism and we call it the \emph{Galois closure of $X$ over $Y$}.
\end{fact} 

\begin{fact}[\cite{bourbaki-alg-comm}, V \S2.2, Corollaire \`a Th.~2]\label{bourbaki-lemma}
%Let $p:X\to Y$ be a Galois cover of (algebraic) schemes with group $G$. Let $L$ be a field and suppose that morphisms 
%$x_1,x_2:\spec(L)\to X$ satisfy $px_1=px_2$. Then there exists a $g\in G$ such that $x_2=gx_1$.
Suppose $G$ is a finite group acting on a ring $A$.  Let $f_1$ and $f_2$ be two homomorphisms from $A$ to a field $L$ with the same restriction to $A^G$. Then there exists a $g\in G$ such that $f_2=f_1 g$.
\end{fact}

\begin{corollary}\label{cor-to-bbk-lemma}
Let $p:X\to Y$ be a Galois cover with group $G$, and $\phi_1,\phi_2:Z\to X$ two morphisms from an integral scheme $Z$ satisfying $p\phi_1=p\phi_2$. Then there exists a $g\in G$ such that $\phi_2=g \phi_1$.
\end{corollary}
\begin{proof}
We may assume that $X=\spec(A)$, and that $\phi_i$ is associated to $f_i:A\to A$, $i=1,2$. Denote by $j$ the inclusion of $A$ in its fraction field. The previous Fact applied to $j f_1$ and $j f_2$ yields a $g\in G$ such that $j f_2=j f_1 g$. Since $j$ is injective, we deduce that $f_2=f_1g$, as required.
\end{proof}

\begin{corollary}\label{map-of-gal-gp}
Suppose we have a commutative diagram 
$$
 \begin{tikzpicture} 
 [cross line/.style={preaction={draw=white, -,
line width=3pt}}]
\matrix(m)[matrix of math nodes, minimum size=1.7em,
inner sep=0pt,
row sep=1.5em, column sep=1em, text height=1.5ex, text depth=0.25ex]
 { 
                        |(P)|{X} 	& |(2)| {X'}          \\
                       |(1)|{Y}             & |(h)|{Y'} \\};
\path[->,font=\scriptsize,>=to, thin]
(P) edge node[left]{$p$} (1) edge node[above]{$f$} (2)
(1) edge node[above]{$h$}  (h)
(2) edge node[right]{$p'$} (h)
;
\end{tikzpicture}
$$
where $p$ and $p'$ are Galois covers with groups $G$ and $G'$. Then we have a homomorphism 
$
\lexp{f}{()}:G\to G'
$
such that, for $g\in G$,
$$
\lexp{f}{g}\, f=f\, g.
$$
\end{corollary}
\begin{proof}
For $g\in G$, we have that $p'fg=hpg=hp=p'f$. Thus, by \ref{cor-to-bbk-lemma}, there is a unique element $g'\in G'$ such that $g'f=fg$. It is readily verified that the assignment $g\mapsto \lexp{f}{g}=g'$ is a homomorphism.
\end{proof}

\begin{definition}[{\cite[IX.2.1]{sga1}}]\label{def-submer}
A morphism  $f:S'\to S$ is \emph{submersive} if it is surjective and makes $S$ into a quotient topological space of $S'$,
i.e., a subset $U$ of $S$ is open if and only if $f^{-1}(U)$ is open in $S'$. A morphism is \emph{universally submersive} if every base change of it remains submersive. 
\end{definition} 
We will exploit the fact, proved in loc.~cit., that a faithfully flat quasi-compact morphism is universally submersive through the following.

\begin{lemma}\label{loc-submer}
Let $f:X\to Y$ be a dominant morphism of finite type of integral schemes whose generic fibre is geometrically integral.
Then there is an open dense subset $U$ of $Y$ such that $f^{-1}(U)\to U$ is faithfully flat (so universally submersive) with geometrically integral fibres.
\end{lemma}
\begin{proof} Let $U_1$ be on open dense subset of $Y$ such that $f$ is flat over $U_1$ by generic flatness.
Since $f$ is dominant, there exists an open dense set $U_2$ such that $f$ is surjective above $U_2$. 
By the constructibility of
the property of being geometrically integral \cite[Tag~0553, Tag~0574]{dejong-stacks}, %http://stacks.math.columbia.edu/tag/0553  ---or EGA???
% for geometrically irreducible use Tag 0553, for geom reduced 0574
we can find an open dense set $U_3$ such that over $U_3$, $f$ has geometrically integral fibres. Then let $U=U_1\cap U_2\cap U_3$.
\end{proof}

\begin{lemma}\label{dirimgal}
Let $f:X\to Y$ be an universally submersive morphism %of (algebraic) schemes 
with geometrically connected fibres and assume $Y$ is connected (whence it follows that $X$ is connected).  The base change functor $f^*:V\mapsto V\times_YX$ from the category of 
\'etale covers of $Y$ to the category of \'etale covers of $X$ is fully faithful and it has a left adjoint $f_*$, i.e., for every \'etale cover $Z\to X$ we have a morphism $Z\to f_*Z$ inducing the natural isomorphism
$$
\Hom_Y(f_*Z,V)=\Hom_X(Z,f^*V),
$$ 
for every \'etale cover $V\to Y$.

Moreover, $f^*$ and $f_*$ take Galois covers to Galois covers and every Galois cover $Z\to X$ yields an exact sequence 
$$
1\to \Gal(Z/f^*f_*Z)\to \Gal(Z/X)\to \Gal(f_*Z/Y)\to 1.
$$
\end{lemma}

Let $Z\to X$ be an \'etale cover as above. Note that the required $f_*Z$ is the solution to the following universal problem. We need to show that there exists an \'etale cover $W\to Y$ completing the diagram
$$
 \begin{tikzpicture} 
 [cross line/.style={preaction={draw=white, -,
line width=3pt}}]
\matrix(m)[matrix of math nodes, minimum size=1.7em,
inner sep=0pt,
row sep=1em, column sep=2.8em, text height=1.5ex, text depth=0.25ex]
 { 
 	 |(3)|{Z}  & 			&			\\
                       & |(P)|{f^*W}	& |(2)| {W}          \\[2em]
                       &|(1)|{X}             & |(h)|{Y} \\};
\path[->,font=\scriptsize,>=to, thin]
(P) edge  (1) edge (2)
(1) edge  (h)
(2) edge[cross line]  (h)
(3) edge (1) edge[cross line] (2) edge (P)
;
\end{tikzpicture}
$$
which is maximal in the sense that for any other \'etale cover $V\to Y$ which fits into an analogous diagram (i.e., $Z\to X$ dominates $f^*V$), $W\to X$ dominates $V\to X$. It will then follow that $W$ is unique up to isomorphism and we will denote it by $f_*Z$.

\begin{proof}
The base change functor from \'etale covers of $Y$ to \'etale covers of $X$ induced by $f$  is fully faithful by \cite[IX.3.4]{sga1}. % and so is our $f^*$.

If we choose a geometric point $\bar{x}$ in $X$ mapping onto $\bar{y}$ in $Y$, it was proved in \cite[IX.5.6]{sga1} that the homomorphism 
$$\pi_1(f):\pi_1(X,\bar{x})\to \pi_1(Y,\bar{y})$$ of \'etale fundamental groups is surjective. 

Let us introduce some abstract notation associated with profinite group actions on finite sets. Given an epimorphism  $\phi:\pi\to \pi'$  of profinite groups with kernel $K$, %=\ker{\phi}$, 
a $\pi$-set $E$ and a $\pi'$-set $E'$, we write
\begin{enumerate}
\item $\phi^*E'$ for the set $E'$ endowed with a $\pi$-action via $\phi$;
\item $\phi_*E$ for the set $E/K$ endowed with a natural $\pi'$-action.
\end{enumerate}
There is an obvious natural bijection
\begin{equation}\label{adj}
\Hom_{\pi'}(\phi_*E,E')\simeq\Hom_{\pi}(E,\phi^*E').\tag{\dag}
\end{equation}

Using this notation, given an \'etale cover $Z\to X$, we define $f_*Z$ to be the \'etale cover of $Y$ which
corresponds via \ref{groth-gal} to the $\pi_1(Y,\bar{y})$-set $\pi_1(f)_*F_{\bar{x}}(Z)$, i.e., to be the cover satisfying the property
$$
F_{\bar{y}}(f_*Z)\simeq \pi_1(f)_*F_{\bar{x}}(Z).
$$

On the other hand, if $V\to Y$ is an \'etale cover, we have that $F_{\bar{x}}(f^*V)\simeq F_{\bar{y}}(V)$ and $f^*V$ clearly corresponds to the $\pi_1(X,\bar{x})$-set $\pi_1(f)^*F_{\bar{y}}(V)$, so the required adjunction is a formal consequence of (\ref{adj}) and \ref{groth-gal}.

%If $N$ is an open normal subgroup of $\pi_1(X,\bar{x})$ corresponding to the Galois cover $Z\to X$, then its image is an open normal subgroup of $\pi_1(Y,\bar{y})$ and thus it corresponds to some Galois cover $W\to Y$. 

If $V\to Y$ is Galois, if follows that $f^*V$ is Galois since it is connected. If $Z\to X$ is Galois, then $\pi_1(X,\bar{x})$ acts transitively on $F_{\bar{x}}(Z)$, hence $\pi_1(Y,\bar{y})$ acts transitively on $\pi_1(f)_*F_{\bar{x}}(Z)=F_{\bar{y}}(f_*Z)$ and $f_*Z\to Y$ is Galois.

Note that the full faithfulness of $f^*$ yields that for every Galois cover $V\to Y$,  $\Gal(f^*V/X)\simeq \Gal(V/Y)$.  The exact sequence follows
from the particular case $V=f_*Z$.

Given the rather indirect flavour of the above proof making use of the theory the \'etale fundamental group and descent, let us give a direct construction of $W$ under the additional hypothesis that $X$, $Y$ and $Z$ are normal and $X\to Y$ faithfully flat. The assumptions imply that $\kk(X)$ is a regular extension of $\kk(Y)$, and we let $W$ be the normalisation of $Y$ in the relative algebraic closure $L$ of $\kk(Y)$ in $\kk(Z)$, which is verifiably  Galois. Then $X\times_YW$ is the normalisation of $X$ in $\kk(X)L$, and it suffices to check that $X\times_YW\to X$ is \'etale, which will subsequently imply that $W\to Y$ is finite \'etale Galois by faithfully flat descent, as required.

This is in fact a consequence of a more general principle stating that, given a tower $Z\to X'\to X$ of finite morphisms between normal connected schemes with $Z\to X$ \'etale and $Z\to X'$ surjective, the morphism $X'\to X$ is necessarily \'etale. Indeed, let us replace $Z$ with its Galois closure over $X$ and perform a base change of the whole situation via $Z\to X$.  Exploiting the fact that $Z\times_XZ\simeq Z\times G$, and restricting attention to its components, 
%and that $Z\to X'$ is surjective, and restricting attention to a component of $Z\times_XZ$, 
we can reduce to the situation where $Z\to X$ is an isomorphism. It follows that $X'\to X$ is a bijective finite morphism of normal schemes and thus an isomorphism.
\end{proof}

\subsection{Definition and basic properties of direct Galois covers}

\begin{definition}\label{def-dir-gal}
An \emph{almost direct Galois cover} is a $\sadir$-morphism  $p:(X,\Sigma)\to(Y,T)$ such that:
\begin{enumerate}
\item $X_i/Y_i$ is a Galois cover with group $G_i$, i=0,1;
\item $G_0$ acts on $\Sigma$ on the left so that $T\simeq G_0\backslash\Sigma$ via $\rlexp{()}{p}$.
\end{enumerate}
\end{definition}
%\item $\aut_{\sadir}((X,\Sigma)/(Y,\sigma))\simeq G_1$;
%A \emph{direct Galois cover} is an almost direct Galois cover in $\sdir$.

\begin{remark}\label{induced-gp-maps}
Using \ref{map-of-gal-gp}, we see that for the above data
\begin{enumerate}
\item there exists a homomorphism $\lexp{\pi_1}{()}:G_1\to G_0$ such that,  for $g_1\in G_1$,
$$\pi_1 g_1=\lexp{\pi_1}{g_1 }\,\pi_1;$$

\item\label{zvizda} for every $\sigma\in \Sigma$, there exists a homomorphism $\lexp{\sigma}{()}:G_1\to G_0$ such that, for $g_1\in G_1$, $$\sigma g_1=\lexp{\sigma}{g_1}\,\sigma.$$
\end{enumerate}
\end{remark}

If $(X,\Sigma)/(Y,\sigma)$ is an almost direct Galois cover, the \emph{almost direct Galois group} comprises the collection of data
$$
(G_1,G_0,\lexp{\pi_1}{()},\tilde{\Sigma}),
$$
where $\tilde{\Sigma}=\{\lexp{\sigma}{()}:\sigma\in\Sigma\}$. On the other hand, the following lemma shows that it is reasonable to informally say that the almost direct Galois group is simply $(G_1,\tilde{\Sigma})$.

\begin{lemma}
Let $p:(X,\Sigma)\to(Y,T)$ be an almost direct Galois cover. Then
\begin{enumerate}
\item\label{jen-dirgal} $\Aut_{\sadir}((X,\Sigma)/(Y,T))\simeq G_1$, and
\item\label{dva-dirgal} geometric fibres of $p$ are $G_1$-orbits.
%\item fibres of \sharp-points?
\end{enumerate}
\end{lemma}
\begin{proof}
For (\ref{jen-dirgal}), we need to show that every $g_1\in G_1$ induces a $\sadir$-automorphism of $(X,\Sigma)$. Writing $g_0=\rlexp{g_1}{\pi_1}$, the condition $\pi_1 g_1=g_0\pi_1$ already gives the first half of the relevant diagram. Now, for each $\sigma\in\Sigma$ and $g=(g_0,g_1)$, we define 
$$\lexp{g}{\sigma}=g_0\sigma g_1^{-1}\stackrel{\ref{induced-gp-maps}(\ref{zvizda})}=g_0\,\lexp{\sigma}{(g_1^{-1})}\sigma\in\Sigma.$$ 

By the commutativity of the diagram
$$
 \begin{tikzpicture}
[cross line/.style={preaction={draw=white, -,
line width=4pt}}]
\node (x1) at  (1.5,2) {$X_1$}; 	\node(x0s) at (3.5,2) {$X_0$};
\node (y1) at (0,1.2) {$X_1$};		\node(y0s) at (2.0,1.2) {$X_0$};	
\node (s1) at (1,0) {$Y_1$};		\node(s0s) at (3.0,0){$Y_0$};	
\path[->,font=\scriptsize,>=to, thin]
(x1) edge node[above,pos=0.5]{${\sigma}$} (x0s) 
	edge node[left,pos=0.3]{$g_1$} (y1)  edge node[right,pos=0.6]{$p_1$}  (s1)
(x0s) edge node[left,pos=0.3]{$g_0$} (y0s) edge node[right,pos=0.5]{$p_0$} (s0s)
(y0s) edge node[right,pos=0.2]{$p_0$} (s0s)
(y1)  edge [cross line] node[above,pos=0.5]{$\lexp{g}{\sigma}$} (y0s)  edge  node[left,pos=0.6]{$p_1$} (s1)
(s1) edge node[above]{$\lexp{p}{\sigma}$} (s0s) 
;
\end{tikzpicture}
$$ 
each $\lexp{g}{()}$ defines a map $\Sigma\to\Sigma$ and for every $\sigma\in\Sigma$, 
$$\lexp{g}{\sigma}\,g_1=g_0\sigma,$$ 
so $g_0$ and $g_1$ give rise to an automorphism $g$ of $(X,\Sigma)$.

To show (\ref{dva-dirgal}), let $y\in (Y,T)(F,\varphi)$ be a point with values in an algebraically closed difference field $(F,\varphi)$. Writing $\pi_1 y_1=y_0$  we have that $\tau y_1=y_0\varphi$ for some $\tau\in T$. Since $X_1/Y_1$ is finite (Galois), there exists a point $x_1\in X_1(F)$ such that $p_1(x_1)=y_1$. Let $x_0=\pi_1(x_1)$ so that $p_0(x_0)=y_0$. Let $\sigma\in\Sigma$ be such that $\rlexp{\sigma}{p}=\tau$ and consider the $\varsigma$-linear points $\sigma x_1$ and $x_0\varphi$ of $X_0$. Since $X_0/Y_0$ is Galois, and
$$
p_0\sigma x_1=\tau p_1 x_1=\tau y_1=y_0\varphi=p_0x_0\varphi,
$$
there exists a $g_0\in G_0$ with $g_0\sigma x_1=x_0\varphi$ and we conclude that $x\in X(F,\varphi)$.

Clearly for every $g\in G$, $gx$ also maps to $y$, so by part~(\ref{jen-dirgal}) we conclude that fibres of $X(F,\varphi)\to Y(F,\varphi)$ are $G_1$-orbits. Moreover, the fibres of $X^\sharp(F,\varphi)\to Y^\sharp(F,\varphi)$ are $\rlexp{(G_1)}{\pi_1}$-orbits.
\end{proof}

%\begin{remark}
%direct Galois cover vs categorical $\sadir$-quotients vs Galois covers of difference schemes.??????????????????????  refer to Appendix
%\end{remark}

We refer the reader interested in the comparison of direct Galois covers with Galois covers of difference schemes defined in \cite{ive-tch}
to Remark~\ref{dirgal-vs-gal}.

\begin{remark}\label{critdirgalcov}
Suppose we have a $\sadir$-morphism $(X,\tilde{\sigma})\to(Y,\sigma)$ such that $X_i/Y_i$ is a Galois cover %of integral schemes 
with group $G_i$, $i=0,1$. %, and that we have a homomorphism $()^{\pi_1}:G_1\to G_0$ such that $\pi_1 g_1=g_1^{\pi_1}\pi_1$, for $g_1\in G_1$. 
Let $\Sigma=G_0\tilde{\sigma}$. Then $(X,\Sigma)\to(Y,\sigma)$ is an almost direct Galois cover.
\end{remark}

\subsection{Local substitutions}

\begin{definition}
An object $(X,\Sigma)$ of $\sadir$ is \emph{faithful} if $\Sigma$ acts faithfully on geometric points of $X$ in the sense that, for every algebraically closed difference field $(F,\varphi)$, $\bar{x}\in X(F,\varphi)$, $\sigma,\sigma'\in\Sigma$, 
$\sigma \bar{x}_1=\sigma'\bar{x}_1$ implies $\sigma=\sigma'$.
\end{definition}

\begin{lemma}\label{faithf-lemma}
Suppose $(Y,T)$ is faithful, and that $p:(X,\Sigma)\to(Y,T)$ is a directly \'etale almost direct Galois cover. Then $(X,\Sigma)$ is also faithful.
\end{lemma}
\begin{proof}
Let $\bar{x}$ be a geometric point on $X$ with $p(\bar{x})=\bar{y}$ and suppose $\sigma\bar{x}_1=\sigma'\bar{x}_1$. Then 
$\rlexp{\sigma}{p}\bar{y}_1=\rlexp{\sigma'}{p}\bar{y}_1$, so the faithfulness of $Y$ implies that $\rlexp{\sigma}{p}=\rlexp{\sigma'}{p}$ and there is a $g_0\in G_0$ such that $\sigma'=g_0\sigma$, so the original relation can be rewritten as $g_0\sigma\bar{x}_1=\sigma\bar{x}_1$. Since $X_0/Y_0$ is \'etale, it follows that $g_0=1$.
\end{proof}

\begin{remark}
Using the previous lemma, if $(X,\Sigma)\to (Y,\sigma)$ is a directly \'etale Galois cover, then $(X,\Sigma)$ is automatically faithful.
\end{remark}

\begin{definition}\label{dir-loc-sub}
Let $(X,\Sigma)/(Y,T)$ be a directly \'etale Galois cover with group $(G,\tilde{\Sigma})$ and $(Y,T)$ faithful. Let $(F,\varphi)$ be an algebraically closed difference field and let
$x,x'\in X(F,\varphi)$, $y\in Y(F,\varphi)$ with $x,x'\mapsto y$. The \emph{local $\varphi$-substitution at $x$} is the unique (by~\ref{faithf-lemma}) $\varphi_x\in\Sigma$ such that $\varphi_x x_1=x_0\varphi$ (i.e., $\varphi_x=\varphi^x$). Since $X_1/Y_1$ is Galois, there exists a $g\in G$
such that $x'=gx$ and 
$$\rlexp{\varphi_x}{g}\,x_1'=\rlexp{\varphi_x}{g}\, g_1x_1=g_0\varphi_xx_1=g_0x_0\varphi=x_0'\varphi,$$
so we conclude that $\varphi_{x'}=\rlexp{\varphi_x}{g}$ and we can define the \emph{local $\varphi$-substitution at $y$} as the $G$-conjugacy class 
$\varphi_y$ of any $\varphi_x$ in $\Sigma$ with $x\mapsto y$.
\end{definition}

\begin{remark}
Suppose $(X,\Sigma)\to(Y,\sigma)$ is a directly \'etale almost direct Galois cover and let us fix a $\tilde{\sigma}\in\Sigma$ so that $\Sigma=G_0\tilde{\sigma}$. Given $x\in X(F,\varphi)$, we can consider the unique $\dot{\varphi}_x\in G_0$ such that $\varphi_x=\dot{\varphi}_x\tilde{\sigma}$, i.e., $\dot{\varphi}_x\tilde{\sigma}x_1=x_0\varphi$.  If $x,x'\mapsto y$, there is a $g\in G$ such that $x'=gx$ and
$$
\varphi_{x'}=\rlexp{\varphi_x}{g}=\rlexp{(\dot{\varphi}_x\tilde{\sigma})}{g}=g_0\dot{\varphi}_x\tilde{\sigma}g_1^{-1}=g_0\dot{\varphi}_x\rlexp{(g_1^{-1})}{\tilde{\sigma}}\tilde{\sigma},
$$
so we conclude that $\dot{\varphi}_{x'}=\rlexp{g_1}{\pi_1}\dot{\varphi}_x\rlexp{(g_1^{-1})}{\tilde{\sigma}}$, a $\rlexp{()}{\tilde{\sigma}}$-conjugate of $\dot{\varphi}_x$. It is therefore meaningful to define $\dot{\varphi}_y$ as the $(G,\rlexp{()}{\tilde{\sigma}})$-conjugacy class in $G_0$ of any $\dot{\varphi}_x$ with $x\mapsto y$.
\end{remark}

\subsection{Constructions of direct Galois covers}\label{ss:constr-dirgal}

\begin{proposition}[Pushforward of a direct Galois cover]\label{pfwd-submr}
Let $f:(X,\sigma)\to (Y,\sigma)$ be a morphism of directly integral almost direct presentations which is directly universally submersive with geometrically connected fibres, and let $(Z,\Sigma)\to(X,\sigma)$ be an almost direct Galois cover. 
For every $\tau\in\Sigma$, there is a diagram
$$
 \begin{tikzpicture}
[cross line/.style={preaction={draw=white, -,
line width=4pt}}]
\matrix(m)[matrix of math nodes, row sep=.09em, column sep=1em, text height=1.2ex, text depth=0.25ex]
{
|(a)|{Z_0}	&				& |(b)|{Z_1}	&			 	& |(e)|{Z_{0\varsigma}} 	&			\\[.6cm]
		&|(A)|{f_{0*}Z_0} 	&			& |(B)|{f_{1*}Z_1}	&					&|(E)|{f_{0\varsigma*}Z_{0\varsigma}}\\[.3cm]       
|(c)|{X_0}	&				& |(d)|{X_1} 	&				& |(f)|{X_{0\varsigma}} 	&			\\[.6cm]
		&|(C)|{Y_0} 		&			& |(D)|{Y_1}		&					&|(F)|{Y_{0\varsigma}}   \\};
\path[->,font=\scriptsize,>=to, thin,inner sep=1pt]
(b)edge node[pos=0.5,above]{$\pi_1$}(a)
(b)edge %node[pos=0.3, right%=-2pt]{$\tau^\sigma$}
(d) 
(b) edge node [pos=0.5,above]{$\pi_2(\tau)$} (e)
(a) edge %[dashed] %node[pos=0.5,left]{$(\tau^\sigma)^{(\rho^\sigma)}=(\tau^\rho)^\sigma$}
(c)
(d)edge node[pos=0.75,above%=-2pt
]{$\pi_1$}(c)
(d) edge node [pos=0.75,above]{$\pi_2(\sigma)$}  (f)
(e) edge (f) edge (E)
(f) edge node[above]{$f_{0\varsigma}$}(F)

(B)edge[dashed,cross line]  %node[pos=0.75,above]{$\pi_1$}
(A) edge[dashed,cross line]  %node[pos=0.25,above]{$\pi_2(\tau???)$} 
(E)
(B)edge[cross line]  %node[pos=0.5,right]{$\tau$}
(D)
(E) edge (F)
(A)edge[cross line] % node[pos=0.62,right%=-2pt]{$\tau^\rho$}
(C)
(D)edge[cross line]  node[pos=0.5,below]{$\pi_1$}(C) edge[cross line]  node[pos=0.5,below]{$\pi_2$} (F)
(a)edge node[pos=0.5,above right]{}(A)
(c)edge node[pos=0.5,above right]{$f_0$}(C)
(b)edge node[pos=0.5,above right]{}(B)
(d)edge node[pos=0.5,above right]{$f_1$}(D);
\end{tikzpicture}
$$ 
which makes $f_*Z=(f_{0*}Z_0,f_{1*}Z_1,{f_{0\varsigma*}Z_{0\varsigma}})$ into an almost direct Galois cover of $Y$.
\end{proposition}
\begin{proof}
While the solid arrows in the diagram come out directly from the assumptions, the dashed arrows are constructed using the universal property of direct images of Galois covers from \ref{dirimgal}. % and we finish by \ref{critdirgalcov}.

Indeed, let $V_1$ be (a component of) $f_{0*}Z_0\times_{Y_0}Y_1$, so that there is a morphism $Z_1\to V_1$. Since $V_1$ is an \'etale cover of $Y_1$, by maximality of $f_{1*}Z_1$, there is a morphism $f_{1*}Z_1\to V_1$, and we take the composite with the natural morphism $V_1\to  f_{0*}Z_0$.

Similarly, for each $\tau\in\Sigma$ we obtain a morphism $f_{1*}Z_1\to f_{0\varsigma*}Z_{0\varsigma}$, as required.
\end{proof}

\begin{proposition}[Direct Galois closure]\label{dirgalcl}
Let $(X,\sigma)\to (Y,\sigma)$ be a directly finite \'etale morphism in $\sadir$ between directly integral objects. There exists an object $(\tilde{X},\tilde{\Sigma})$ in $\sadir$, an inclusion $\iota:\mathring{\Sigma}\hookrightarrow\tilde{\Sigma}$ inducing a morphism $\mathring{X}=(\tilde{X},\mathring{\Sigma})\to(\tilde{X},\tilde{\Sigma})$ as in the diagram
$$
\begin{tikzpicture} 
\matrix(m)[matrix of math nodes, row sep=0.1em, column sep=1.5em,text height=1.3ex, text depth=0.25ex]
{       
 |(u1)|{\mathring{X}}		&[1.4em] |(u2)|{\tilde{X}}	\\[1.5em]
 |(d1)|{X} 					& 								\\
 								& |(b)|{Y}            					\\}; 
\path[->,font=\scriptsize,>=to, thin]
(u1) edge  (u2) 
(u1) edge %[dashed]  
(d1)
(d1) edge (b)
(u2) edge %[dashed] 
(b);
\end{tikzpicture}
$$ 
where the vertical arrows are almost direct Galois covers, which is minimal in the sense that any other almost direct Galois cover that fits into an analogous diagram directly dominates $\tilde{X}$.

%makes $\mathring{\tilde{X}}\to X$ and $\tilde{X}\to Y$ into almost direct Galois covers.
\end{proposition}
The above diagram is called the \emph{almost direct Galois closure} of $(X,\sigma)\to(Y,\sigma)$. Note that this is consistent with the notion of Galois closure in difference algebraic geometry \cite{ive-tgs}.

\begin{proof}
Let $\tilde{X}_0$ be the Galois closure of $X_0$ over $Y_0$. The fibre product
$$
\bar{X}_1=X_1\times_{Y_0\times Y_0}\tilde{X}_0\times\tilde{X}_0
$$
is an \'etale cover of $Y_1$ %and therefore an \'etale cover of $Y_1$, 
with a transitive action of $\Gal(\tilde{X}_0/Y_0)\times\Gal(\tilde{X}_0/Y_0)$  on its connected components $C_1,\ldots,C_r$.

%Let $\
%and let $Z$ be a least  connected \'etale cover of $Y$ dominating all of its components ($Z$ is constructed as a connected component of the fibre product of the components of $X_1'$ over $Y$). 
We let $\tilde{X}_1$ be the Galois closure of $C_1$ over $Y_1$, so we obtain a correspondence 
$\tilde{X}_0\stackrel{\pi_1}{\leftarrow}\tilde{X}_1\stackrel{\tilde{\sigma}}{\rightarrow}\tilde{X}_0$ and a $\sadir$-morphism $(\tilde{X},\tilde{\sigma})\to(X,\sigma)$.

%Most of the construction is already described in the statement. Since $\Gal(\tilde{X}_1/Y_1)$ is in fact a decomposition subgroup of $\tilde{X}_1$ in the Galois cover $$(\tilde{X}_0\times\tilde{X}_{0\varsigma})\times_{X_0\times X_{0\varsigma}}X_1\to X_1$$ with group
%$$\Gal(\tilde{X}_0\times\tilde{X}_{0\varsigma}/Y_0\times Y_{0\varsigma})\simeq\Gal(\tilde{X}_0/Y_0)\times\Gal(\tilde{X}_{0\varsigma}/Y_{0\varsigma}),$$ 
%the first projection gives a suitable homomorphism 
%$\Gal(\tilde{X}_1/Y_1)\to \Gal(\tilde{X}_0/Y_0)$ and 

Writing $\mathring{\Sigma}=\Gal(\tilde{X}_0/X_0)\tilde{\sigma}$ and 
$\tilde{\Sigma}=\Gal(\tilde{X}_0/Y_0)\tilde{\sigma}$, 
Remark~\ref{critdirgalcov} shows that $(\tilde{X},\mathring{\Sigma})\to(X,\sigma)$ and $(\tilde{X},\tilde{\Sigma})\to(Y,\sigma)$ are almost direct Galois covers. 

Suppose $(Z,\Sigma_Z)$ is an almost direct Galois cover of $(Y,\sigma)$ such that for some $\tilde{\sigma}\in\Sigma_Z$, we have a morphism $(Z,\tilde{\sigma})\to(X,\sigma)$. Since $\tilde{X}_0$ is a Galois closure of $X_0$ over $Y_0$, we have a morphism $Z_0\to\tilde{X}_0$ and, by the universal property of fibre products, we get a morphism $Z_1\to \bar{X}_1$, where $Z_1$ actually lands onto a component $C_j$ of $\bar{X}_1$. By transitivity of the Galois action on the components, there is a morphism $g_j: C_j\to C_1$, so we get a composite \'etale cover $Z_1\to C_j\to  C_1$, whence a morphism $Z_1\to\tilde{X}_1$. We leave the verification that everything commutes to the reader.  
\end{proof}

\section{Galois formulae and first-order formulae}\label{s:gal-fmla}

\subsection{Direct Galois stratifications and direct Galois formulae}

\begin{definition} 
Let $(R,\varsigma)$ be a difference domain,  and let $(X,\sigma)$ be a direct variety in $\sadir_{(R,\varsigma)}$.
A \emph{normal almost direct Galois stratification} 
$$ 
\cA=\langle X, Z_i/X_i, C_i\, |\, i\in I\rangle
$$
of $(X,\sigma)$ over $(R,\varsigma)$ is a partition of $(X,\sigma)$
into a finite set of directly integral normal locally closed  subvarieties %$\sadir$-subobjects 
$(X_i,\sigma)$ of $(X,\sigma)$,
each equipped with a directly connected almost direct Galois covering $(Z_i,\Sigma_i)/(X_i,\sigma)$ with group 
$(G_i,\tilde{\Sigma}_i)$, and   
$C_i$ is a $G_i$-conjugacy domain in $\Sigma_i$. %, as in 
%\cite[Section~\ref{sect:galois}]{ive-tch}.
%union of conjugacy classes in $\Sigma_i$.
\end{definition}
 
\begin{definition}
Let $\cA$ be an almost direct Galois stratification on $(X,\sigma)$ over $(R,\varsigma)$. Then $\cA$ defines a `point set' subassignment $\cA^\flat$ of 
$X^\flat$ as follows. 
For a point $s\in\spec^\varsigma(R)$ and
an algebraically closed
 difference field $(F,\varphi)$ extending $(\kk(s),\sigma^s)$, 
% and a $(K,\varphi)$-valued point $s$ in $(S,\sigma)$, 
$$
\cA^\flat(s,(F,\varphi))=\cA_s(F,\varphi)=\bigcup_i\{x\in X_{i,s}(F,\varphi)\,\, |\,\, \varphi^{Z_i/X_i}_x\subseteq C_i\},
$$
where $\varphi^{Z_i/X_i}_x$ denotes the local $\varphi$-substitution at $x$, as defined in \ref{dir-loc-sub}.

The \emph{almost direct Galois formula} over $(R,\varsigma)$ associated with %the above stratification 
$\cA$ is defined as the `realisation' subassignment $\cA^\sharp$ of $X^\sharp$ by the rule
%Given a point $s\in S$ and
%an algebraically closed
% difference field $(F,\varphi)$ extending $(\kk(s),\sigma^s)$, 
%% and a $(K,\varphi)$-valued point $s$ in $(S,\sigma)$, 
$$
\cA^\sharp(s,(F,\varphi))=\cA_s^\sharp(F,\varphi)=\{x_0\in X^\sharp(F,\varphi):x\in\cA_s(F,\varphi)\},
%\bigcup_i\{x_0\in X_{i,s}^\sharp(F,\varphi)\,\, |\,\, \varphi^{Z_i/X_i}_x\subseteq C_i\}.
$$
%where $\varphi^{Z_i/X_i}_x$ denotes the local $\varphi$-substitution at $x$, as defined in \ref{dir-loc-sub}.
so that we can think of $\cA_s^\sharp(F,\varphi)$ as the projection along $\pi_1$ of $\cA_s(F,\varphi)$.
%\pi_{1,s}\{x_1: x\in\cA_s(F,\varphi)\}=\{x_0:x\in\cA_s(F,\varphi)\}$.
\end{definition}

\begin{notation}
In informal discussion we may omit the word `almost' from the above terms and refer simply to \emph{direct Galois stratifications} and \emph{direct Galois formulae}. 
\end{notation}

\begin{remark}
If we fix a lift $\sigma_i\in\Sigma_i$ of $\sigma$ for each $i$, the
above data is equivalent to fixing for each $i$ 
 a $\rlexp{()}{\sigma_i}$-conjugacy domain $\dot{C}_i$ in $G_i$, i.e., a union of $\rlexp{()}{\sigma_i}$-conjugacy classes in  $G_i$. Clearly,
 $$
 \cA_s^\sharp(F,\varphi)=\bigcup_i\{x_0\in X_{i,s}^\sharp(F,\varphi)\,\, |\,\, \dot{\varphi}_{x}^{Z_i/X_i}\subseteq \dot{C}_i\}.
$$ 
\end{remark}

\begin{definition}\label{inflation-refinement}
Let $(X,\sigma)$ be a direct variety in $\sadir_{(R,\varsigma)}$ and let
$\cA=\langle X,Z_i/X_i,C_i\rangle$ be an almost direct Galois stratification
on $X$.
\begin{enumerate}[wide]
\item Suppose that for each $i$ we have an almost direct  
covering $(Z_i',\Sigma_i')/(X_i,\sigma)$ which dominates 
$(Z_i,\Sigma_i)/(X_i,\sigma)$. Let $\pi_i:\Sigma_i'\to\Sigma_i$ denote
the associated surjective map. The \emph{inflation} of $\cA$ is defined
as 
$$\cA'=\langle X,Z'_i/X_i,\pi_i^{-1}(C_i)\rangle.$$
It has the property that for every $s\in S$, and every algebraically closed
$(F,\varphi)$ extending $(\kk(s),\sigma^s)$,
$$
\cA'_s(F,\varphi)=\cA_s(F,\varphi).
$$
\item Suppose that we have a further stratification of $X_i$ into finitely many
directly integral normal locally closed %$(k,\sigma)$-difference 
subschemes $X_{ij}$.
For each $i,j$, let $Z_{ij}$ be a direct component of $Z_{i}\times_{X_i}X_{ij}$, and 
let $$D(Z_{ij,1})=\{g\in\Aut(X_{ij,1}\times_{X_{i,1}}Z_{i,1}/X_{ij,1})\simeq\Gal(Z_{i,1}/X_{i,1}):g(Z_{ij,1})=g(Z_{ij,1})\}$$ be the decomposition group 
of $Z_{ij,1}$. Since $\pi_1(Z_{ij,1})\subseteq Z_{ij,0}$, it follows that $\rlexp{()}{\pi_1}$ takes $D(Z_{ij,1})$ into $D(Z_{ij,0})$.
Consider $\Sigma_{ij}=\{\tau\in\Sigma_i:\tau(Z_{ij,1})\subseteq Z_{ij,0}\}$ (which is non-empty since $Z_i/X_i$ is Galois) and the inclusion
$\iota_{ij}:\Sigma_{ij}\hookrightarrow\Sigma_i$.  It can be verified, by \ref{critdirgalcov} for example, that each $(Z_{ij},\Sigma_{ij})/(X_{ij},\sigma)$ is a Galois cover with group $D(Z_{ij,1})$.

%let $\p_{ij}\in\spec^\sigma(A)$ be the ideal corresponding to the generic point of 
%$X_{ij}$ and let $\q_{ij}$
%be an element of $\spec^\Sigma(C_i)$ extending $\p_{ij}$.
%Let $G_{ij}$ be the stabiliser of $\q_{ij}$ in $G$, let 
%$\Sigma_{ij}=\{\sigma\in\Sigma_i:\sigma^{-1}(\q_{ij})=\q_{ij}\}$, and write
%$\iota_{ij}:\Sigma_{ij}\hookrightarrow\Sigma_i$, $C_{ij}=C_i/\q_{ij}$.
%Then each $(C_{ij},\Sigma_{ij})/(X_{ij},\sigma)$ is a ring/scheme covering
%with group $(G_{ij},\tilde{\Sigma}_{ij})$.

The \emph{refinement} of $\cA$ associated to the above data is defined
as
$$\cA'=\langle X,Z_{ij}/X_{ij},\iota_{ij}^{-1}(C_i)\rangle.$$
It has the property that for every $s\in S$, and every algebraically closed
$(F,\varphi)$ extending $(\kk(s),\sigma^s)$,
$$
\cA'_s(F,\varphi)=\cA_s(F,\varphi).
$$
%\item With notation of \cite[\ref{def:pullbk}]{ive-tch}, the \emph{pullback} $f^*\cB$ of $\cB$ with respect to
%$f$ is defined
%as a refinement of
%$$\langle X,Z_{j}/X_{j},\iota_{j}^{-1}(D_j)\rangle$$
%to a normal refinement of the stratification $X_j$ of $X$. 
%It has the property that for every $s\in S$, and every algebraically closed
%$(F,\varphi)$ extending $(\kk(s),\sigma^s)$,
%$$
%f^*\cB_s(F,\varphi)=f_s^{-1}(\cB_s(F,\varphi)).
%$$
\end{enumerate}
\end{definition}

\begin{definition} 
Let $(X,\sigma)$ be a direct variety in $\sadir_{(R,\varsigma)}$.
The class of  $(R,\varsigma)$-Galois formulae on $X$ has a \emph{Boolean algebra} 
structure as follows. 
\begin{enumerate}
\item $\bot_X=\langle X,X/X,\emptyset\rangle$, $\top_X=\langle X,X/X,\{\sigma\}\/\rangle$.
\end{enumerate}
For Galois formulae on $X$ given by $\cA$ and $\cB$, upon a refinement
and an inflation we may assume that $\cA=\langle X,Z_i/X_i,C_i\rangle$ and
$\cB=\langle X,Z_i/X_i,D_i\rangle$, with $C_i,D_i\subseteq\Sigma_i$.
\begin{enumerate}[resume]
\item $\cA\land\cB=\langle X,Z_i/X_i,C_i\cap D_i\rangle$.
\item $\cA\lor\cB=\langle X,Z_i/X_i,C_i\cup D_i\rangle$.
\item $\lnot\cA=\langle X,Z_i/X_i,\Sigma_i\setminus C_i\rangle$.
\end{enumerate}
\end{definition}

\subsection{First-order formulae}

\begin{definition}
Let $(R,\varsigma)$ be a transformal domain.
\begin{enumerate}[wide]
\item
A \emph{first-order formula 
over $(R,\varsigma)$} is 
a first-order formula $\theta(x_1,\ldots,x_n;a_1,\ldots,a_m)$ 
in the language of difference rings with free variables
$x_1,\ldots,x_n$ and parameters $a_1,\ldots,a_m$ from $R$.

\item An $(R,\varsigma)$-formula $\theta(x_1,\ldots,x_n;a_1,\ldots,a_m)$ 
gives rise to
a subassignment $\theta^\sharp$ of $\Af^{n\sharp}_{(R,\varsigma)}$ by the following 
procedure. For any  $s\in\spec^\varsigma(R)$ and any
difference
field $(F,\varphi)$ extending $(\kk(s),\varsigma^s)$, writing $\bar{s}$ for the composite $(R,\varsigma)\to (\kk(s),\varsigma^s)\to(F,\varphi)$,
the value of
$$\theta^\sharp(s,(F,\varphi))=\theta_s(F,\varphi)$$ is the set of realisations of the formula $\theta(x_1,\ldots,x_n,\bar{s}(a_1),\ldots,\bar{s}(a_m))$ in $(F,\varphi)$. %it is a subset of $F^n$).

\item An $(R,\varsigma)$-subassignment $\cF$ of $\Af^n_R$ is called \emph{definable} if
there is a first-order formula $\theta(x_1,\ldots,x_n)$ over $(R,\sigma)$
such that $\cF=\theta^\sharp$.
\end{enumerate}
\end{definition}

\subsection{Existentially closed difference fields}

It is known (\cite{angus}, \cite{zoe-udi}) that
the first-order theory of difference fields has a model-companion called ACFA, which axiomatises existentially closed difference fields.

An axiom scheme for ACFA is obtained by a first-order transliteration of the following statement (the crucial statement is known as `axiom H').
\begin{fact}\label{axiom-H}
An algebraically closed inversive difference field $(F,\varphi)$ is existentially closed if and only if
%\begin{quote}
every H-direct $(X,\sigma)$ in $\sdir_{(F,\varphi)}$ satisfies $X^\sharp(F,\varphi)\neq\emptyset$. 
%\end{quote}
\end{fact} 

The following is a uniform variant, obtained by using \ref{loc-improvement} to make all fibres geometrically H-direct, and subsequently applying  axiom H.

\begin{corollary}
Let $(R,\varsigma)$ be a transformal domain, and let $(X,\sigma)$ be a direct variety in $\sdir_{(R,\varsigma)}$ whose generic fibre over  $R$ is  geometrically H-direct. Then there exists a $\varsigma$-localisation $R'$ of $R$ such that, for every $s\in\spec^\varsigma(R')$ and every existentially closed $(F,\varphi)$ extending $(\kk(s),\varphi^s)$, 
$$
X_s^\sharp(F,\varphi)\neq\emptyset.
$$
\end{corollary}

\subsection{Fields with powers of Frobenius} 
\begin{notation}
In the sequel, $p$ denotes a rational prime, and $q$ is a power of $p$. 

If $q=p^n$, the map
$$
\varphi_q:\bar{\F}_p\to\bar{\F}_p, \ \ \ \varphi_q(\alpha)=\alpha^q
$$ 
is the $n$-th power of the Frobenius automorphism on the algebraic closure of $\F_p$.

If $k$ is a finite field, we may also write 
$$\varphi_k=\varphi_{|k|}:\bar{k}\to\bar{k}.$$
\end{notation}

In \cite{udi}, Hrushovski proves that ACFA is in fact the elementary theory of difference fields $(\bar{\F}_p,\varphi_{q})$. The crucial ingredient is the following consequence of his twisted Lang-Weil estimate.

\begin{fact}[\cite{udi}]\label{twist-LW}
Let $(R,\varsigma)$ be a transformal domain of finite $\varsigma$-type over $\Z$, and let $(X,\sigma)$ be 
a direct variety in $\sdir_{(R,\varsigma)}$ whose generic fibre over $R$ is geometrically H-direct. 
Then there exists a $\varsigma$-localisation $R'$ of $R$ and an integer $N>0$ such that for every $s\in\spec^\varsigma(R')$ and every field $(\bar{\F}_p,\varphi_q)$ extending $(\kk(s),\varphi^s)$ with $q\geq N$, 
$$
X_s^\sharp(\bar{\F}_p,\varphi_q)\neq\emptyset.
$$ 
\end{fact}

\subsection{Equivalent subassignments and theories}

\begin{definition}
Let $(R,\varsigma)$ be a %normal 
transformal domain and let $(X,\sigma)$ be an object of $\sadir_{(R,\varsigma)}$.
Let $\cF$ and $\cF'$ be $(R,\varsigma)$-subassignments of $X$ or $X^\sharp$.
\begin{enumerate}[wide]
\item We shall say that $\cF$ and $\cF'$ are \emph{equivalent} over $(R,\varsigma)$ and write 
$$
\cF\equiv_{(R,\varsigma)}\cF',
$$
if for every closed $s\in\spec^\varsigma(R)$, every algebraically closed difference field $(F,\varphi)$ extending $(\kk(s),\sigma^s)$, 
$$
\cF(s,(F,\varphi))=\cF'(s,(F,\varphi)).
$$

\item We shall write 
$$
\cF\equiv_{(R,\varsigma)}^{\acfa}\cF',
$$
if the above holds when $(F,\varphi)$ ranges over suitable existentially closed difference fields. Additionally, we write
$$
\cF\equiv_{(R,\varsigma),\gen}^{\acfa}\cF',
$$
if $\cF\equiv_{(R',\varsigma)}^{\acfa}\cF'$ for some finite $\varsigma$-localisation $R'$ of $R$.

\item When $(R,\varsigma)$ is of finite $\varsigma$-type over $\Z$, and $N$ a positive integer, 
we shall %say that $\cF$ and $\cF'$ are \emph{equivalent with respect to fields with Frobenii} over $(R,\varsigma)$ and 
write
$$
\cF\equiv_{(R,\varsigma)}^{\frob,N}\cF',
$$
if for every closed $s\in\spec^\varsigma(R)$, every finite field $k$ with $(\bar{k},\varphi_k)$ extending $(\kk(s),\sigma^s)$ and $|k|>N$, 
$$
\cF(s,(\bar{k},\varphi_k))=\cF'(s,(\bar{k},\varphi_k)).
$$
\item If $(R,\varsigma)$ is of finite $\varsigma$-type over $\Z$, we write 
$$\cF\equiv_{(R,\varsigma)}^{\frob,\infty}\cF',$$ if there is an $N>0$ such that $\cF\equiv_{(R,\varsigma)}^{\frob,N}\cF'$.
We write
$$\cF\equiv_{(R,\varsigma),\gen}^{\frob,\infty}\cF',$$ if $\cF\equiv_{(R',\varsigma)}^{\frob,\infty}\cF',$ for some finite $\varsigma$-localisation $R'$ of $R$.
\end{enumerate}
\end{definition}

\begin{definition}\label{theories}
Let $(R,\varsigma)$ be a transformal domain. 
\begin{enumerate}
\item The theory $$\acfa_{(R,\varsigma)}$$
is the set of first-order sentences $\theta$ over $(R,\varsigma)$ such that for any $s\in\spec^\varsigma(R)$ and any existentially closed difference field $(F,\varphi)$ extending $(\kk(s),\varsigma^s)$, we have
$(F,\varphi)\models\theta_s$. We write
$$\acfa_{(R,\varsigma),\gen}$$ for the union of theories $\acfa_{(R',\varsigma)}$, 
where $(R',\varsigma)$ ranges over all finite $\varsigma$-localisations of $R$.

\item If $(R,\varsigma)$ is of finite $\varsigma$-type over $\Z$, the theory $$T_{(R,\varsigma)}^\infty=T_{(R,\varsigma)}^{\frob,\infty}$$
is the set of first-order sentences $\theta$ over $(R,\varsigma)$ such that there exists a positive integer $N$ such that for every closed $s\in\spec^\varsigma(R)$, every finite field $k$ with $(\bar{k},\varphi_k)$ extending $(\kk(s),\varsigma^s)$ and $|k|>N$, we have
$(\bar{k},\varphi_k)\models\theta_s$. 
We write
$$
T^{\infty}_{(R,\varsigma),\gen}
$$
for the union of theories $T^{\infty}_{(R',\varsigma)}$, where $R'$ ranges over all finite 
$\varsigma$-localisations of $R$.
\end{enumerate}
\end{definition}

\begin{notation}\label{class-theories}
We write
\begin{enumerate}
\item $\acfa_0=\acfa_{(\Q,\id)}$ for the theory of existentially closed fields of characteristic zero;
\item $\acfa_p=\acfa_{(\F_p,\id)}$ for the theory of existentially closed fields of positive characteristic $p$;
\item $T_0^\infty=T^{\infty}_{(\Z,\id),\gen}$ for the set of sentences true in fields $(\bar{\F}_p,\varphi_q)$ for all but finitely many $p$ and all sufficiently large $q$;
\item $T^\infty_p=T^\infty_{(\F_p,\id)}$ for the set of sentences true in $(\bar{F}_p,\varphi_q)$ for all large enough $q$.
\end{enumerate}
\end{notation}

\subsection{Logic quantifier elimination}
\begin{fact}[{\cite[1.6]{zoe-udi}}]\label{log-qe}
Every formula $\psi(x)$ in the variables $x=(x_1,\ldots,x_n)$ is equivalent modulo ACFA to a disjunction of formulae of the form
$\exists y \theta(x,y)$ where $y$ is a single variable, $\theta$ is quantifier-free, and in every model $(F,\varphi)$, for every $a\in F$, 
$\theta(a,b)$ implies that $b$ is algebraic over the subfield generated by $a, \varphi(a),\ldots, \varphi^m(a)$ for some $m$. 

In the above terminology, let $(k,\varsigma)$ be a prime field (either $(\Q,id)$ or $(\F_p,id)$). Then 
$$\psi(x)^\sharp\equiv^{{\acfa}}_{(k,\varsigma)}(\bigvee_i\exists y\theta_i(x,y))^\sharp$$
for $\theta_i$ as above.
\end{fact}

\subsection{First-order formulae associated with Galois formulae}

\begin{remark}\label{dirgal-fo}
Let $(X,\sigma)$ be an almost direct presentation. 
The previously studied subassignments of $X$ and $X^\sharp$ fit in the hierarchy of definable subassignments as follows.
\begin{enumerate}[wide]
\item The subassignment $X$ itself corresponds to a (positive) difference quantifier-free definable subset of the algebraic variety $X_1$.
\item The subassignement $X^\sharp$ itself corresponds to an existentially definable subset of the algebraic variety $X_0$, since it is a projection of $X$ via $\pi_1$, see \ref{dirpts}.
\item An almost direct Galois formula on $X$ is $\equiv$-equivalent to a definable set of the form that appears upon the logic quantifier elimination~\ref{log-qe} down to $\exists_1$-formulae.

%\begin{remark}\label{gal-2-fo}
Indeed, suppose $(Z,\Sigma)/(X,\sigma)$ is an almost direct  Galois cover with group $(G,\tilde{\Sigma})$,  fix a $\tilde{\sigma}\in\Sigma$
and a $(G,\rlexp{()}{\tilde{\sigma}})$-conjugacy domain $\dot{C}$. Then
\begin{multline*}
\langle(Z,\Sigma)/(X,\sigma),\dot{C}\rangle^\sharp(F,\varphi)\\=\{x_0\in X_0(F)\,\, |\,\, \exists z\in Z_1(F)\ \  \pi_1(z)\mapsto x_0\  \land\ 
\bigvee_{g_0\in\dot{C}}z\in(Z,g_0\tilde{\sigma})(F,\varphi)\},
\end{multline*}
so, in view of the fact that the conditions in the above disjunction are quantifier-free,  it is clear that a basic direct Galois formula is equivalent to an existential first-order formula of a particular shape. 

On the other hand, a general direct Galois formula is just a positive Boolean combination of  basic ones so the result follows.
% we refer the reader to \ref{dirgal-fo} for further discussion.

In case $Z$ and $X$ are direct, the associated first-order formula can be made even more explicit. The data yields a closed immersion 
$$
 \begin{tikzpicture} 
\matrix(m)[matrix of math nodes, row sep=2em, column sep=1.5em, text height=1.5ex, text depth=0.25ex]
 {
 			& |(3)|{Z_1} 			& 		\\
 |(l2)|{Z_0}	& |(l3)|{Z_0\times Z_0} 	& |(l4)|{Z_0}\\
 }; 
\path[->,font=\scriptsize,>=to, thin]
(3) edge node[above]{${\pi_1}$} (l2) edge node[above]{${\tilde{\sigma}}$}   (l4)
(l3) edge (l2) edge (l4) ;
\path[right hook->,>=to,thin] (3) edge (l3);
\end{tikzpicture}
$$
whence
\begin{multline*}
\langle(Z,\Sigma)/(X,\sigma),\dot{C}\rangle^\sharp(F,\varphi)\\=\{x_0\in X_0(F)\,\, |\,\, \exists z_0\in Z_0(F),\ \  z_0\mapsto x_0\  \land\ 
\bigvee_{g_0\in\dot{C}}(z_0,g_0^{-1}z_0\varphi)\in Z_1\}.
\end{multline*}
% that appears after the known logic quantifier-elimination of \cite{angus} and \cite{zoe-udi}. We will show \ldots ?????
%Comment that $X^\sharp$ is essentially a difference qf-definable subset of the algebraic variety $X_0$, and $\cA^\sharp$ is a more general definable subset of $X_0$.
% 
\end{enumerate}
\end{remark}

The goal of subsequent sections is to show that existentially closed difference fields (and, asymptotically, fields with powers of Frobenius) allow quantifier elimination for Galois formulae and every first order formula is equivalent to a Galois formula over such fields.

\section{Direct image theorems}\label{s:dirim-th}

\subsection{Direct images of Galois formulae}

\begin{definition}
Let $f:(X,\sigma)\to (Y,\sigma)$ be a morphism in $\sadir_{(R,\varsigma)}$ and let $\cA$ be an almost direct Galois stratification on $X$.
%associated with a Galois formula $\chi(x;s)\equiv\{x\in X_s\ |\ \ar(x)\subseteq\con(\cA)\}$. 
We define a subassignment $f_{\exists}\cA$ of $Y$ and a subassignment $f_{\exists}^\sharp\cA$ of $Y^\sharp$ by the following rule.
For $s\in S$ and
$(F,\varphi)$ an algebraically closed difference field extending $(\kk(s),\sigma^s)$,
\begin{enumerate}
\item
$f_{\exists}\cA(s,(F,\varphi))=(f_{\exists}\cA)_s(F,\varphi)=f_s(\cA_s(F,\varphi))\subseteq Y_s(F,\varphi)$;
\item
$
f_{\exists}^\sharp\cA(s,(F,\varphi))=(f_{\exists}^\sharp\cA)_s(F,\varphi)=
f_{0,s}(\cA_s^\sharp(F,\varphi))\subseteq Y_s^\sharp(F,\varphi).
$
\end{enumerate}
%It can also be considered
%as an expression 
%$$\upsilon(y;s)\equiv\{y\in Y_s\ |\ \exists x\ \chi(x;s), f_s(x)=y\}$$
%which justifies the notation somewhat.
\end{definition}

\begin{lemma}\label{galois-towers}
Suppose that a diagram %$(Z,\Sigma_Z)\to(X,\Sigma_X)\to(Y,\sigma)$
$$
\begin{tikzpicture} 
\matrix(m)[matrix of math nodes, row sep=0.1em, column sep=1.5em,text height=1.3ex, text depth=0.25ex]
{       
 |(u1)|{(Z,\mathring{\Sigma})}		&[.5em] |(u2)|{(Z,\Sigma)}	\\[1.5em]
 |(d1)|{(X,\sigma)} 					& 								\\
 								& |(b)|{(Y,\sigma)}            					\\}; 
\path[->,font=\scriptsize,>=to, thin]
(u1) edge  (u2) 
(u1) edge %[dashed]  
(d1)
(d1) edge node[above]{$f$} (b)
(u2) edge %[dashed] 
(b);
\end{tikzpicture}
$$
%
%
%
%
%$$
%\begin{tikzpicture} 
%\matrix(m)[matrix of math nodes, row sep=0.3em, column sep=0.1em,text height=1.3ex, text depth=0.25ex]
%{       
% |(u1)|{(Z,\Sigma_Z)}		&[1.4em] |(u2)|{}	\\[1.0em]
% |(d1)|{(X,\Sigma_X)} 		& |(d2)|{(Y,\sigma)} 		\\}; 
%\path[->,font=\scriptsize,>=to, thin]
%(u1) edge  (d2)
%(u1) edge  (d1)
%(d1) edge node[below]{$f$} (d2);
%\end{tikzpicture}
%$$
consists of almost direct Galois covers $(Z,\mathring{\Sigma})/(X,\sigma)$ and
$(Z,\Sigma)/(Y,\sigma)$, where the top horizontal arrow is induced by an inclusion
$\iota:\mathring{\Sigma}\hookrightarrow\Sigma$ and $f$ is directly finite \'etale. 
Let $C\subseteq\mathring{\Sigma}$ be a $\Gal(Z/X)$-conjugacy domain, and let $\iota_*C\subseteq \Sigma$ be the $\Gal(Z/Y)$-conjugacy domain induced by $C$.
Then
$$
f_{\exists}^\sharp\langle Z/X,C\rangle\equiv\langle Z/Y,\iota_*C\rangle^{\sharp}.
$$
\end{lemma}
\begin{proof}
We will show more, that 
$f_{\exists}\langle Z/X,C_0\rangle=\langle Z/Y,C\rangle^{\flat}$.
For the left-to-right inclusion, let $y\in Y(F,\varphi)$ be such that there exists an $x\in\langle Z/X,C\rangle(F,\varphi)$ with $f(x)=y$. Thus, there exists a $z\in (Z,\mathring{\Sigma})(F,\varphi)$ with $z\mapsto x$ and $\varphi_z\in C$, so that its $\Gal(Z/X)$-conjugacy class $\varphi_x\subseteq C$. But $\varphi_y$ is the $\Gal(Z/Y)$-conjugacy class of $\varphi_z$ so it is clearly contained in $\iota_*C$.

For the other inclusion, suppose $y\in\langle Z/Y,C\rangle(F,\varphi)$ for an algebraically closed difference field $(F,\varphi)$. There exists a $z\in (Z,\Sigma)(F,\varphi)$ such that $z\mapsto y$ and $$\varphi_z\in \iota_*C=\bigcup_{g\in\Gal(Z/Y)}\lexp{g}{C},$$ 
%=\bigcup_{g_1\in\Gal(Z_1/X_1)}g_1^{\pi_1}C_0g_1^{-1},$$
so let $\varphi_z\in \lexp{g}{C}$ for some $g\in\Gal(Z/Y)$. Then $\varphi_{g^{-1}z}\in C$, so the image $x$ of $g^{-1}z$ in $X$ witnesses $f(x)=y$ and $x\in\langle Z/X,C\rangle(F,\varphi)$.
\end{proof}

\begin{proposition}\label{dirim-finet}
Let $f:(X,\sigma)\to(Y,\sigma)$ be a directly finite \'etale morphism in $\sadir$, let $(Z,\Sigma)/(X,\sigma)$ be an almost direct Galois cover, and let $C\subseteq\Sigma$ be a conjugacy domain. Let $(\tilde{Z},\mathring{\Sigma})$, $(\tilde{Z},\tilde{\Sigma})$, $\iota:\mathring{\Sigma}\hookrightarrow\tilde{\Sigma}$ be the data associated with the direct Galois closure of $Z$ over $Y$. %, so that $\mathring{Z}=\iota^*\tilde{Z}$. 
Let $\mathring{C}$ be the preimage of $C$ under the surjection $\mathring{\Sigma}\to\Sigma$, and let $\iota_*\mathring{C}$ be the least conjugacy domain in $\tilde{\Sigma}$ containing $\iota(\mathring{C})$.
Then
$$
f_{\exists}^\sharp\langle Z/X,C\rangle\equiv\langle\tilde{Z}/Y,\iota_*\mathring{C}\rangle^\sharp.
$$
\end{proposition}

\begin{proof}
As in the previous proof, let us show that $f_{\exists}\langle Z/X,C\rangle\equiv\langle\tilde{Z}/Y,\iota_*\mathring{C}\rangle^\flat$ already. The diagram 
$$
\begin{tikzpicture} 
 [cross line/.style={preaction={draw=white, -,line width=3pt}}]
\matrix(m)[matrix of math nodes, row sep=0.1em, column sep=1.0em,text height=1.3ex, text depth=0.25ex]
{       
					&[0em]|(h1)|{\mathring{Z}} &			&|(h2)|{\tilde{Z}}\\[.3em]
 |(u1)|{Z} 	& 			& 	&		\\[1.5em]
 |(d1)|{X} 				&			& 					&		\\
 					&			& |(b)|{Y}            		&		\\}; 
\path[->,font=\scriptsize,>=to, thin]
(h1) edge (u1) edge (d1) edge (h2)
(h2) edge (b)
(u1)  edge [cross line] (b)
(u1) edge  (d1)
(d1) edge node[below,pos=0.4]{$f$}(b)
;
\end{tikzpicture}
$$
%$$
%\begin{tikzpicture} 
% [cross line/.style={preaction={draw=white, -,line width=3pt}}]
%\matrix(m)[matrix of math nodes, row sep=0.1em, column sep=1.0em] %,text height=1.3ex, text depth=0.25ex]
%{       
%					&[1em]|(h1)|{\mathring{Z}} &					&[1em]|(h2)|{Z}\\[.3em]
% |(u1)|{\mathring{\tilde{X}}} 	& 			&[1.0em] |(u2)|{\tilde{X}}	&		\\[1.5em]
% |(d1)|{X} 				&			& 					&		\\
% 					&			& |(b)|{Y}            		&		\\}; 
%\path[->,font=\scriptsize,>=to, thin]
%(h1) edge (u1) edge (d1) edge (h2)
%(h2) edge (u2) edge (b)
%(u1) edge [cross line]  (u2) edge [cross line] node[above right]{$\mathring{f}$} (b)
%(u1) edge node[left]{$r$} (d1)
%(d1) edge node[below]{$f$}(b)
%(u2) edge node[pos=0.3,right]{$\tilde{f}$} (b);
%\end{tikzpicture}
%$$
shows the situation described in the statement, where we wrote $\mathring{Z}$ for $(\tilde{Z},\mathring{\Sigma})$. 

We have
$$
f_\exists\langle Z/X,C\rangle\stackrel{\text{(inflation)}}{\equiv}
f_\exists\langle \mathring{Z}/X,\mathring{C}\rangle\stackrel{(\ref{galois-towers})}{\equiv}
\langle Z/Y,\iota_*\mathring{C}\rangle^\flat.
$$
%
%
%
%upon the contruction 
%of the direct Galois closure of $X$ over $Y$ using \ref{dirgalcl}.
%Starting with a $\Gal(\mathring{Z}/X)$-conjugacy class $C$ in $\mathring{\Sigma}_Z$, we can consider it as a $\Gal(\mathring{Z}/\mathring{X})$-conjugacy domain $C'$, and write $C''$ for $C'$ considered in $Z/\tilde{X}$, so that $\tilde{C}=\iota_*C$ is the $\Gal(Z/Y)$-conjugacy domain induced by $C''$.
%Using \ref{galois-towers}, we have that $\tilde{f}_{\exists}\langle Z/\tilde{X},C''\rangle\equiv\langle Z/Y,\tilde{C}\rangle$ and 
%$r_{\exists}\langle \mathring{Z}/\mathring{\tilde{X}},C'\rangle\equiv\langle\mathring{Z}/X,C\rangle$.
%Thus
%$$
%f_{\exists}\langle\mathring{Z}/X,C\rangle\equiv f_{\exists}r_{\exists}\langle \mathring{Z}/\mathring{\tilde{X}},C'\rangle\equiv \mathring{f}_{\exists}\langle \mathring{Z}/\mathring{\tilde{X}},C'\rangle\equiv \tilde{f}_{\exists}\langle Z/\tilde{X},C''\rangle\equiv \langle Z/Y,\tilde{C}\rangle^{\flat}.
%$$
\end{proof}

\begin{proposition}\label{dirim-geom-conn}
Let $f:(X,\sigma)\to(Y,\sigma)$ be a morphism of directly integral varieties in $\sadir_{(R,\varsigma)}$ which is directly universally submersive with geometrically connected fibres, and let $(Z,\Sigma)\to(X,\sigma)$ be an almost direct Galois cover and let $C\subseteq\Sigma$ be a conjugacy domain. Let $f_*Z$ be the almost direct Galois cover of $Y$ obtained in \ref{pfwd-submr}.
Then there is a finite $\varsigma$-localisation $R'$ of $R$ such that
$$
f_{\exists}^\sharp\langle Z/X,C\rangle\equiv_{(R',\varsigma)}^{{\acfa}}\langle f_*Z/Y,f_*C\rangle^\sharp,
$$
where $f_*C$ denotes the image of $C$ under the surjective map $\Sigma\to\Sigma_{f_*Z}$.
\end{proposition}
\begin{proposition}\label{dirim-geom-conn-frob}
When $(R,\varsigma)$ is of finite $\varsigma$-type over $\Z$, there exists a positive integer $N$ and a finite $\varsigma$-localisation $R'$ of $R$ such that we have the analogous statement of \ref{dirim-geom-conn} with
$$
f_{\exists}^\sharp\langle Z/X,C\rangle\equiv_{(R',\varsigma)}^{\frob,N}\langle f_*Z/Y,f_*C\rangle^\sharp,
$$
\end{proposition}

\begin{proof}[Proof of \ref{dirim-geom-conn}.]
Writing $W=f_*Z$, the diagram
$$
 \begin{tikzpicture} 
 [cross line/.style={preaction={draw=white, -,
line width=3pt}}]
\matrix(m)[matrix of math nodes, minimum size=1.7em,
inner sep=0pt,
row sep=1em, column sep=2.8em, text height=1.5ex, text depth=0.25ex]
 { 
 	 |(3)|{Z}  & 			&			\\
                       & |(P)|{f^*W}	& |(2)| {W}          \\[2em]
                       &|(1)|{X}             & |(h)|{Y} \\};
\path[->,font=\scriptsize,>=to, thin]
(P) edge  (1) edge (2)
(1) edge  (h)
(2) edge[cross line]  (h)
(3) edge (1) edge[cross line] (2) edge (P)
;
\end{tikzpicture}
$$
resulting from \ref{pfwd-submr} shows (by inflation, \ref{inflation-refinement}) 
that $\langle Z/X,f^*f_*C\rangle=\langle f^*W/X,f_*C\rangle$, and so $f_\exists \langle Z/X,C\rangle=f_\exists\langle Z/X,f^*f_*C\rangle=f_\exists\langle f^*W/X,f_*C\rangle$. Thus, it is sufficient to prove that, for any conjugacy domain $D\subseteq \Sigma_W$, 
$$
f_{\exists}\langle f^*W/X,D\rangle\equiv\langle W/Y,D\rangle.
$$
Indeed, let $y\in Y(F,\varphi)$ with $\varphi_y\in D$. There exists an $y'\in W(F,\varphi)$ with $y'\mapsto y$ and $\varphi_{y'}\in D$, i.e., 
there exists a $\tau\in D$ such that $y'\in W^\tau(F,\varphi)$, $\tau y_1'=y_0'\varphi$. By construction, the fibre $(f^*W_{y'},\tau)$ is directly geometrically integral, so by Axiom H (\ref{axiom-H}), there exists an $x'\in (f^*W_{y'},\tau)(F,\varphi)$ such that its image $x\in X(F,\varphi)$ witnesses $\varphi_x\subseteq D$, $f(x)=y$.
\end{proof}
The proof of \ref{dirim-geom-conn-frob} is completely analogous, one simply replaces the use of Axiom H by the use of the twisted Lang-Weil estimate \ref{twist-LW}.

\begin{theorem}\label{dirdirim-acfa}
Let $f:(X,\sigma)\to(Y,\sigma)$ be a morphism of direct varieties in $\sadir_{(R,\varsigma)}$, with $(R,\varsigma)$ a transformal domain. Let $\cA$ be an almost direct Galois stratification on $(X,\sigma)$. Then there exists an integer $n$, a localisation $(R',\varsigma)$ of $R_{-n}$ and an almost direct Galois stratification $\cB$ on $(Y,\sigma)$ defined over $R'$ such that
$$
f_{\exists}^\sharp\cA\equiv_{(R',\varsigma)}^{{\acfa}}\cB^\sharp.
$$
%Moreover, if $(Y,\sigma)$ is direct, then $\cB$ can be made direct.
\end{theorem}

\begin{theorem}\label{dirdirim-frob}
When $(R,\varsigma)$ is of finite $\varsigma$-type over $\Z$, there exists a positive integer $N$ such that we have the analogous statement of \ref{dirdirim-acfa} with
$$
f_{\exists}^\sharp\cA\equiv_{(R',\varsigma)}^{\frob,N}\cB^\sharp.
$$
\end{theorem}

\begin{proof}[Proof of \ref{dirdirim-acfa} and \ref{dirdirim-frob}.]
Upon a direct irreducible decomposition and a localisation, we may assume that $f$ is a morphism of H-direct normal objects, and that $\cA$ is given as $\langle(Z,\Sigma)/(X,\sigma),C\rangle$ for an almost direct Galois cover $Z/X$. 

We begin by performing a direct `baby' Stein factorisation as follows. Let $L_i$ be the relative algebraic closure of $\kk(Y_i)$ in $\kk(X_i)$ for $i=0,1$,  so that $L_{0\varsigma}$ is the relative algebraic closure of $\kk(Y_{0\varsigma})$ in $\kk(X_{0\varsigma})$.
% and let $L_1$ be the relative algebraic closure of $\kk(Y_1)$ in $\kk(X_1)$. 
Clearly $L_0\hookrightarrow L_1\hookleftarrow L_{0\varsigma}$, and let 
$\tilde{Y}_i$ be the normalisation of $Y_i$ in $L_i$ for $i=0,1$,  so that $\tilde{Y}_{0\varsigma}$ is the normalisation of $Y_{0\varsigma}$ in $L_{0\varsigma}$. The resulting diagram
$$
 \begin{tikzpicture}
[cross line/.style={preaction={draw=white, -,
line width=4pt}}]
\matrix(m)[matrix of math nodes, row sep=.9em, column sep=.5em, text height=1.5ex, text depth=0.25ex]
{			& |(x0)| {X_0}	&				& |(x1)| {X_1} 	&			& |(x0s)| {X_{0\varsigma}}	\\   [.2em]
|(y0)|{\tilde{Y}_0} &			& |(y1)|{\tilde{Y}_1} 	&			&  |(y0s)| {\tilde{Y}_{0\varsigma}}&			\\  [.4em]
%		& |(d4)| {X}	&			& |(d3)| {Y}	\\   %[.8cm]
			& |(s0)|{Y_0} 		&			& |(s1)|{Y_1} 	&			& |(s0s)|{Y_{0\varsigma}} 				\\};
\path[->,font=\scriptsize,>=to, thin]
(x1) edge node[above]{} (x0) edge node[above]{} (x0s) 
	edge node[left,pos=0.3]{} (y1) edge (s1)
(x0) edge node[left,pos=0.3]{} (y0) edge (s0)
(x0s) edge node[left,pos=0.3]{} (y0s) edge (s0s)
(y0) edge (s0)
(y0s) edge (s0s)
(y1) edge [cross line] node[above,pos=0.2]{} (y0) edge [cross line] node[above,pos=0.8]{} (y0s)  edge (s1)
(s1) edge node[above]{} (s0) edge node[above]{} (s0s) 
;
\end{tikzpicture}
$$ 
gives the required factorisation of $f$, allowing us to reduce the consideration to the following two cases.

By localising, we may assume that the morphism $(X,\sigma)\to(\tilde{Y},\sigma)$  is directly universally submersive with geometrically integral fibres, so we reduce to the known case \ref{dirim-geom-conn}.
The complement is lower dimensional, so we proceed by devissage.

By localising, we may assume that the morphism $(\tilde{Y},\sigma)\to(Y,\sigma)$ is directly finite \'etale, so we finish by \ref{dirim-finet}.
The complement is lower dimensional, so we proceed by devissage.
 
We end up with an almost direct Galois stratification on $Y$. 

%For the moreover case, when $Y$ is direct, we refine it to a direct Galois stratification using \ref{almdir2dir}.
\end{proof}

\begin{corollary}\label{cor-dirdirim}
In addition to assumptions of \ref{dirdirim-acfa}, let $(R,\varsigma)$ be the inversive closure of a transformal domain of finite $\varsigma$-type over a difference field or over $\Z$. Then we have the following.
\begin{enumerate}
\item There exist a finite $\varsigma$-localisation $R'$ of $R$ and an almost direct Galois stratification $\cB$ on $Y$ over $R'$ such that 
$$
f_{\exists}^\sharp\cA\equiv_{(R',\varsigma)}^{\text{\acfa}}\cB^\sharp.
$$
\item If $(R,\varsigma)$ is over $\Z$, there exist a finite $\varsigma$-localisation $R'$ of $R$, an almost direct Galois stratification $\cB$ on $Y$ over $R'$ and a positive integer $N$ such that 
$$
f_{\exists}^\sharp\cA\equiv_{(R',\varsigma)}^{\frob,N}\cB^\sharp.
$$
\end{enumerate}
\end{corollary}

\subsection{Quantifier elimination for direct Galois formulae}

\begin{theorem}\label{dirqe}
Let $(R,\varsigma)$ be the inversive closure of a transformal domain of finite $\varsigma$-type over a difference field or $\Z$. Let $\theta(x)=\theta(x;a)$ be a first order formula in the language of difference rings in variables 
$x=x_1,\ldots,x_n$ with parameters $a$ from $(R,\varsigma)$. Then we have the following.
\begin{enumerate}
\item There exists a direct Galois 
stratification $\cA$ of the difference affine $n$-space over a finite $\varsigma$-localisation $(R',\varsigma)$ of $R$ such that 
$$
\theta^\sharp\equiv_{(R',\varsigma)}^{\text{\acfa}}\cA^\sharp.
$$
\item If $R$ is over $\Z$, there exists a direct Galois 
stratification $\cA$ of the difference affine $n$-space over a finite $\varsigma$-localisation $R'$ of $R$ and a positive integer $N$ such that 
$$
\theta^\sharp\equiv_{(R',\varsigma)}^{\frob,N}\cA^\sharp.
$$
\end{enumerate}
\end{theorem}

\begin{proof}
The proof of the theorem is purely formal from \ref{cor-dirdirim} (where the difficult work was done) by induction on the complexity of $\theta(x)$. The crucial step is the elimination of the existential quantifier, where we apply \ref{cor-dirdirim} to a projection morphism.
It is analogous to the derivation of 3.26 from 3.23 in \cite{ive-tgsacfa} so we omit it. 
\end{proof}

\begin{remark}
The following is a logician's way of interpreting the statements of \ref{dirqe}.
\begin{enumerate}[wide]
\item The class of definable $(R,\varsigma)$-subassignments is $\equiv_{(R,\sigma),\gen}^{\text{\acfa}}$-equivalent to the class of direct Galois formulae over $(R,\varsigma)$. 
%In particular, Galois parameter-free formulae are equivalent to first-order formulae modulo the theory ACFA.
\item   The class of definable $(R,\varsigma)$-subassignments is $\equiv_{(R,\sigma),\gen}^{\frob,\infty}$-equivalent to the class of direct Galois formulae over $(R,\varsigma)$. %In particular, parameter-free Galois formulae are equivalent to first-order formulae modulo the theory $T^\infty$ of first-order sentences true in all difference fields $(\bar{\F}_p,\varphi_q)$, for a sufficiently large $q$.
\end{enumerate} 
\end{remark}

\begin{remark}
It is possible to stratify $\spec^\varsigma(R)$ into locally closed pieces so that the conclusions of \ref{cor-dirdirim} and \ref{dirqe} hold over each piece. In other words, a Galois-type formula can match a first-order formula over \emph{every} point of $\spec^\varsigma(R)$.

Indeed, the assumptions on $(R,\varsigma)$ ensure that it is Ritt, i.e., that $\spec^\varsigma(R)$ is a noetherian topological space with Zariski topology induced from $\spec(R)$.  Theorems \ref{dirdirim-acfa} and \ref{dirdirim-frob} ensure that a suitable $\cB$ can be found on an open dense subset of $\spec^\varsigma(R)$, so we can proceed by noetherian induction on the closed complement. 
\end{remark}

\subsection{Sentences}

\begin{remark}\label{sent-base-field}
Let $\theta$ be a first-order sentence over an inversive difference field $(k,\varsigma)$.
Let $S$ denote the (trivial) direct presentation associated with $\spec(k,\varsigma)$. By \ref{dirqe}, there exists a direct Galois stratification $\cA$ on $S$ so that $\theta$ is equivalent to the Galois formula $\cA^\sharp$. Since $\spec(k)$ is a point, $\cA$ consists of a single direct Galois cover, 
$$
\cA=\langle S,(Z,\Sigma)/(S,\varsigma),C\rangle.
$$

\begin{enumerate}
\item If $C\neq\emptyset$, there exists an existentially closed difference field $(F,\varphi)$ extending $(k,\varphi)$ with $(F,\varphi)\models\theta$, or, equivalently,
$$
\cA^\sharp(F,\varphi)\neq\emptyset.
$$

\item The sentence $\theta$ belongs to the theory $\acfa_{(k,\varsigma)}$ if and only if $C=\Sigma$.
\end{enumerate}
\end{remark}

\begin{proof}
Note that $Z=(Z_0,Z_1)=(\spec(L_0),\spec(L_1))$, where $L_0$ and $L_1$ are finite Galois extensions of $k$ with $L_0\to L_1$. 
 If $\sigma\in\Sigma$, the solid arrows in the diagram
$$
\begin{tikzpicture} 
\matrix(m)[matrix of math nodes, row sep=2.5em, column sep=2.5em,text height=1.3ex, text depth=0.25ex]
{       
 |(u1)|{\bar{k}}		& |(u2)|{\bar{k}}	\\  %[1.5em]
 |(m1)|{L_0} 		& 	|(m2)|{L_1}		\\
  |(d1)|{k} 		& |(d2)|{k}            		\\}; 
\path[->,font=\scriptsize,>=to, thin]
(u1) edge [dashed] node[above]{$\bar{\sigma}$} (u2) 
(m1) edge  node[above]{$\sigma$} (m2) edge (u1)
(d1) edge  node[above]{$\varsigma$} (d2) edge (m1)
(m2) edge (u2)
(d2) edge (m2)
;
\end{tikzpicture}
$$ 
show the given data upon fixing embeddings $L_0\subseteq L_1\subseteq\bar{k}$ into the algebraic closure of $k$. Using the classical extension theorem \cite[Theorem~V.2.8]{lang}, we can lift $\sigma$ to a dashed arrow $\bar{\sigma}$ on $\bar{k}$, and then we can embed $(\bar{k},\bar{\sigma})$ into an existentially closed difference field $(F,\varphi)$. Thus, if $\sigma\in C$, $\cA^\sharp(F,\varphi)\neq\emptyset$, as claimed in (1).

For (2), if $C\neq\Sigma$, then (1) gives an existentially closed difference field $(F,\varphi)$ extending $(k,\varsigma)$ such that $(F,\varphi)\models \lnot\theta\equiv\langle Z/S, \Sigma\setminus C\rangle^\sharp$. 
\end{proof}

\begin{remark}\label{sent-over-Z}
Let $(R,\varsigma)$ be either $(\Z[1/n],\id)$ for some non-zero integer $n$, or $(\F_q,\varphi_r)$ for some powers $q$ and $r$ of a prime $p$. Let $\theta$ be a first-order sentence over $(R,\varsigma)$. % and write $S$ for $\spec(R,\varsigma)$ considered as an object in $\sadir_{(R,\varsigma)}$. 

By \ref{dirqe}, there is a finite localisation $R'$ of $R$, an integer $N>0$ and a basic Galois stratification on $S'=\spec(R',\varsigma)$ as an object in $\sadir$
$$
\cA=\langle S',(Z,\Sigma)/(S',\varsigma),C\rangle.
$$
such that
$$
\theta^\sharp\equiv^{\frob,N}_{(R',\varsigma)}\cA^\sharp.
$$ 
\begin{enumerate}
\item If $C\neq\emptyset$ then for every $\varsigma$-localisation $R''$ of $R'$ there exists a closed point $s\in\spec^\sigma(R'')$ such that for infinitely many finite fields $k$ with $(\bar{k},\varphi_k)$ extending $(\kk(s),\varsigma^s)$ and $|k|>N$, we have $$(\bar{k},\varphi_k)\models\theta_s.$$ % or, equivalently,
%$$
%\cA^\sharp(\bar{k},\varphi_k)\neq\emptyset.
%$$

\item The sentence $\theta$ belongs to the theory $T^{\infty}_{(R,\varsigma),\gen}$ if and only if $C=\Sigma$.
\end{enumerate}
\end{remark}
\begin{proof}
When $(R,\varsigma)=(\F_q,\varphi_r)$, then $R'=R$ and we argue as in \ref{sent-base-field} for $(k,\varsigma)=(\F_q,\varphi_r)$. Since $L_0$ and $L_1$ are finite, it follows that $\sigma$ is a power of Frobenius on $L_0$, and the relevant diagram in \ref{sent-base-field} can be completed by the infinitely many powers of Frobenius $\bar{\sigma}$ which restrict to $\sigma$ on $L_0$.

 In the case $(R,\varsigma)=(\Z[1/n],\id)$, it follows that $R'=\Z[1/n']$ where $n$ divides $n'$. Given that the difference operator on $S'$ is the identity, the Galois cover $(Z,\Sigma)/(S',\id)$ is associated with particularly simple diagrams indexed by $\sigma\in\Sigma$ 
$$
 \begin{tikzpicture} 
\matrix(m)[matrix of math nodes, row sep=2em, column sep=1.5em, text height=1.5ex, text depth=0.25ex]
 {
 |(l2)|{Z_0}	& |(l3)|{Z_1} 	& |(l4)|{Z_0}\\
 			& |(3)|{S'} 			& 		\\
 }; 
\path[->,font=\scriptsize,>=to, thin]
(l2) edge (3)
(l4) edge (3)
%(3) edge  (l2) edge node[above]{${\tilde{\sigma}}$}   (l4)
(l3) edge node[above]{${\pi_1}$} (l2) edge node[above]{${\sigma}$} (l4)
(l3) edge (3);
\end{tikzpicture}
$$
where $Z_0/S$ and $Z_1/S$ are Galois covers with groups $G_0$ and $G_1$. Since $Z_1$ dominates $Z_0$, the map $\lexp{\pi_1}{()}:G_1\to G_0$ is surjective. For each $\sigma\in\Sigma$, \ref{cor-to-bbk-lemma} gives that $\sigma=g_0\pi_1$ for some $g_0\in G_0$.

Hence, if $\sigma\in C$, 
we conclude that $C=C_0\pi_1$ with $C_0$ a conjugacy domain in $G_0$ containing $g_0$. Moreover, finding a point $z\in(Z,\Sigma)(F,\varphi)$ with $\varphi_z=\sigma$ reduces to finding a point $s\in S'$ with local Frobenius substitution with respect to the cover $Z_0/S'$ contained in $C_0$. The classical Chebotarev density theorem gives a non-zero density of primes $s$ with that property, which proves (1).

For (2), if $C\neq\Sigma$, then (1) applied to $\lnot\theta\equiv\langle Z/S,\Sigma\setminus C\rangle$ shows that $\theta\notin T^{\infty}_{(R,\varsigma),\gen}$.

\end{proof}

%\begin{remark}
%\begin{enumerate}
%\item Remark~\ref{sent-base-field} applied to $(k,\varsigma)=(\Q,\id)$ and $(k,\varsigma)=(\F_p,\id)$ gives a decision procedure for the most commonly considered theories $\acfa_0$ and $\acfa_p$ in model theory.
%\item Remark~\ref{sent-over-Z} applied to $(R,\varsigma)=(\Z,\id)$ and $(R,\varsigma)=(\F_p,\id)$ gives a decision procedure for the theories $T^\infty_0$ of sentences true in all difference fields $(\bar{k},\varphi_k)$ of all characteristics for sufficiently large finite field $k$ and $T^\infty_p$ of of sentences true in all difference fields $(\bar{F}_p,\varphi_q)$ for sufficiently large $q$.
%\end{enumerate}
%\end{remark}

\begin{corollary}
With notation of \ref{class-theories}, we have
$$
\acfa_0=T^\infty_0\ \ \ \text{ and }\ \ \ \acfa_p=T^\infty_p.
$$
\end{corollary}
\begin{proof}
A sentence $\theta$ over $\Q$ makes sense over some $\Z[1/n]$. Galois stratification for fields with Frobenius will produce a Galois formula $\cA$ over some $\Z[1/n']$ with $n|n'$ which is equivalent to $\theta$ over fields with high enough power of Frobenius. On the other hand, the stratification procedure for existentially closed difference fields yields the same Galois formula $\cA$. The criteria for a Galois sentence belonging to $\acfa_0$ and $T^\infty_0$ are exactly the same, so we conclude that $\theta\in\acfa_0$ if and only if $\theta\in T^\infty_0$. We argue similarly in characteristic $p$.
\end{proof}

\section{Effective quantifier elimination}\label{s:eff-qe}

\subsection{Effective (direct difference) algebraic geometry}
\begin{definition}
\begin{enumerate}%[wide]
\item A ring $R$ is \emph{primitive recursive}, if (modulo some
G\"odel numbering), $R$ is a primitive recursive set and the operations of addition,
multiplication, multiplicative inverse are primitive recursive functions.

\item A primitive recursive field $k$ has a \emph{splitting algorithm} if there is a primitive recursive algorithm for factoring elements of $k[T]$ into irreducible factors.
\item A primitive recursive field $k$ has \emph{elimination theory} if every finitely generated (explicitly presented) extension of $k$ has a splitting algorithm.

\item A field $k$ is called \emph{effective}, if it is primitive recursive and $k$ is perfect with a splitting algorithm.

\item A domain $R$ is \emph{effective}, if its fraction field $K$ is effective, $R$ is primitive recursive, and $R$ is a primitive recursive subset of $K$.
\end{enumerate}
\end{definition}

\begin{definition}
\begin{enumerate}%[wide]
\item 
An inversive difference ring $(R,\varsigma)$ is \emph{primitive recursive}, if $R$ is a primitive recursive ring, and the difference operator $\varsigma$ and its inverse $\varsigma^{-1}$
are primitive recursive functions.

\item An inversive difference field $(k,\varsigma)$ is called \emph{effective}, if it is primitive recursive and $k$ is effective.

\item An inversive transformal domain $(R,\varsigma)$ is \emph{effective}, if its fraction field $(K,\varsigma)$ is effective, $(R,\varsigma)$ is primitive recursive, and $R$ is a primitive recursive subset of $K$.
\end{enumerate}
\end{definition}

\begin{example}
The transformal domains $(\Z,\id)$, $(\Q,\id)$, $(\F_q,\varphi_p)$ are effective.
\end{example}

\begin{definition}
Let $(R,\varsigma)$ be an effective transformal domain. 
\begin{enumerate}
\item We say that an algebraic variety $V$ over $R$ is \emph{effectively presented} if $V$ is of finite presentation over $R$ and its presentation is explicitly given, and similarly for morphisms.
\item An object $(X,\Sigma)$ of $\vadir_{(R,\varsigma)}$ is \emph{effectively presented} if $X_0$, $X_1$, $X_{0\varsigma}$ and all the morphisms $\pi_1:X_1\to X_0$, $\pi_2(\sigma):X_1\to X_{0\varsigma}$, $\sigma\in\Sigma$, are effectively presented.  We make an analogous definition for morphisms in $\vadir_{(R,\varsigma)}$.
\item A normal almost direct Galois stratification $\cA=\langle X, Z_i/X_i, C_i\, |\, i\in I\rangle$ is \emph{effectively presented} if the base 
$(X,\sigma)$ is effectively presented and all the pieces $Z_i$, $X_i$ are  affine normal and effectively presented.
\end{enumerate}
\end{definition}

\begin{remark}\label{effective-ag}
By \cite[19.2.10]{fried-jarden}, an effective field $k$ has elimination theory and it can serve as a base field for a well-behaved and well-understood \emph{effective/constructive algebraic geometry}. Indeed, by the detailed treatments in \cite{seidenberg}, \cite{fried-jarden}, \cite{orgogozo}, the following operations on effectively presented algebraic varieties over $k$ are known to be primitive recursive:
\begin{enumerate}
\item computing fibre products;
\item decomposing a variety into irreducible components;
\item computing the image of a morphism;
\item computing the relative algebraic closure;
\item computing the loci of flatness/smoothness/\'etaleness/geometrically connected fibres of a morphism;
\item normalisation of a (normal) integral variety in an extension of its function field;
\item computation of Galois groups, Galois closure and decomposition subgroups in a given Galois cover.
\end{enumerate}
Moreover, if the input data for the above algorithms is given over an effective ring $R$, the algorithms can effectively compute an element $f\in R$ so that the output data is defined over the localised ring $R_f$.
\end{remark}

\begin{remark}\label{effective-dag}
Given an effective difference field $(k,\varsigma)$, the operations in $\vadir_{(k,\varsigma)}$ reduce to classical operations on algebraic varieties over $k$, so we automatically obtain a rich framework for \emph{effective direct difference algebraic geometry}.
\end{remark}

\subsection{Effective quantifier elimination for ACFA}

\begin{theorem}\label{effdirqe}
Let $\theta(x)=\theta(x;a)$ be a first order formula in the language of difference rings in variables 
$x=x_1,\ldots,x_n$ with parameters $a$ from an effective difference field $(k,\varsigma)$. 
A primitive recursive procedure can compute an effectively presented direct Galois 
stratification $\cA$ of the difference affine $n$-space over $(k,\varsigma)$ such that 
$$
\theta^\sharp\equiv_{(k,\varsigma)}^{\acfa}\cA^\sharp.
$$
%\item If $R$ is over $\Z$, there exists a Galois 
%stratification $\cA$ of the difference affine $n$-space over $R$ and a positive integer $N$ such that 
%$$
%\theta^\sharp\equiv_{(R,\varsigma)}^{\frob,N}\cA^\sharp.
%$$
\end{theorem}
\begin{proof}
The goal is to show that the algorithm can be described without reference to indefinite loops and unbounded searches, and that various induction proofs can in fact be transformed into procedures using bounded loops.

The outer loop, following the proof of \ref{dirqe}, is bounded by the complexity of $\theta(x)$, and the only nontrivial procedures it invokes are instances of \ref{dirdirim-acfa}, so it will suffice to show that taking direct images of an effective direct Galois stratification via \ref{dirdirim-acfa} is primitive recursive.
%(we avoid the extra loop from \ref{cor-dirdirim} by working over a difference field).

Now, \ref{dirdirim-acfa} is done by induction on dimension, so its main loop is bounded by dimensions of the varieties involved in the direct presentations $(X,\sigma)$ and $(Y,\sigma)$. The next possible problem is a possible jump in the number of direct components produced by the direct decomposition \ref{direct-decomp} on the `bad loci' of lower dimension excised at each step, but Cohn 
\cite[Solution to Problem I*, Chapter 8, no.~14]{cohn} already argued that the procedure is primitive recursive, and Hrushovski even gives explicit bounds for the number of components in terms of degrees of the correspondences involved in \cite[Proposition~2.2.1]{udi-mm}. 

There are no more dangerous control loops to consider, so it suffices to verify that all the algebraic-geometric constructions used in all the constituent steps of the proof of \ref{dirdirim-acfa} are primitive recursive. By inspection, all these operations reduce to the algorithms from \ref{effective-ag} and we are done.
%
%The following is the list of basic primitive recursive operations on algebraic varieties over $k$ which suffice to build up all the constructions used in the proof. They are all known to be primitive recursive, we refer the reader to the detailed treatments in \cite{seidenberg}, \cite{fried-jarden}, \cite{orgogozo}.
%\begin{enumerate}
%\item computing fibre products;
%\item decomposing a variety into irreducible components;
%\item computing the image of a morphism;
%\item computing the relative algebraic closure;
%\item computing the loci of flatness/smoothness/\'etaleness/geometrically connected fibres of a morphism;
%\item normalisation of a (normal) integral variety in an extension of its function field;
%\item computation of Galois groups, Galois closure and decomposition subgroups in a given Galois cover.
%\end{enumerate}
\end{proof}

\begin{corollary}\label{prdec}
Let $(k,\varsigma)$ be an effective difference field. The theory $\acfa_{(k,\varsigma)}$ of existentially closed difference fields extending $(k,\varsigma)$ is decidable by a primitive recursive procedure. 
%Moreover, a primitive recursive procedure can decide whether a first-order sentence %$\theta$ 
%with parameters from $(k,\varsigma)$ is consistent with   ${\rm ACFA}_{(k,\varsigma)}$.
\end{corollary}
\begin{proof}
We repeat the argument of \ref{sent-base-field} in an effective way.
Let $\theta$ be a sentence with parameters in $(k,\varsigma)$.
% and let $$ be the (trivial) direct presentation associated with the spectrum of $(k,\varsigma)$. %(as in \ref{dirpts}). 
Using \ref{effdirqe}, a primitive recursive procedure can compute a direct Galois stratification $\cA$ on $S=\spec(k,\varsigma)\in\sadir_{(k,\varsigma)}$ so that $\theta$ is equivalent to the Galois formula $\cA^\sharp$. Since $\spec(k)$ is a point, $\cA$ consists of a single direct Galois cover, 
$$
\cA=\langle X,(Z,\Sigma)/(X,\sigma),C\rangle.
$$
The sentence $\theta$ is entailed by $\acfa_{(k,\varsigma)}$ if and only if $C=\Sigma$, and this can be checked by a primitive recursive procedure. 

%Note that $\theta$ is consistent with ${\rm ACFA}_{(k,\varsigma)}$ if and only if $C\neq\emptyset$.
\end{proof}

\begin{corollary}
The theories $\acfa_0$ and $\acfa_p$ are primitive recursive decidable.
\end{corollary}

\subsection{Effective quantifier elimination for fields with powers of Frobenius}

\begin{theorem}\label{effdirqe-frob}
Let $\theta(x)$ be a first order formula in the language of difference rings in variables 
$x=x_1,\ldots,x_n$ over an effective transformal domain $(R,\varsigma)$ of finite $\varsigma$-type over $\Z$. 
A primitive recursive procedure can compute an effectively presented direct Galois 
stratification $\cA$ of the difference affine $n$-space over a $\varsigma$-localisation $(R',\varsigma)$ of $R$ and a positive integer $N$ such that 
$$
\theta^\sharp\equiv_{(R',\varsigma)}^{\frob,N}\cA^\sharp.
$$
%\item If $R$ is over $\Z$, there exists a Galois 
%stratification $\cA$ of the difference affine $n$-space over $R$ and a positive integer $N$ such that 
%$$
%\theta^\sharp\equiv_{(R,\varsigma)}^{\frob,N}\cA^\sharp.
%$$
\end{theorem}
\begin{proof}
An essential ingredient of the proof is the effectivity of Hrushovski's bound needed for \ref{twist-LW}. It is argued in \cite{udi} that a primitive recursive procedure can compute a $\varsigma$-localisation of $R$ and an integer $N>0$ so that \ref{twist-LW} holds.
 
The rest of the proof is analogous to \ref{effdirqe}, bearing in mind the effective algebraic geometry over an effective domain in which every operation is done modulo a localisation as in \ref{effective-ag}.
\end{proof}

\begin{corollary}
Let $(R,\varsigma)$ be $(\Z[1/n],\id)$ for a non-zero integer $n$ or $(\F_q,\varphi_r)$ for some prime powers $q$ and $r$.
The theory $T^{\infty,gen}_{(R,\varsigma)}$ is decidable by a primitive recursive procedure. Moreover, given a $\theta\in T^{\infty,gen}_{(R,\varsigma)}$, a primitive recursive procedure can compute the (finite) list of exceptional cases consisting of
\begin{enumerate}
\item points $s\in\spec^\varsigma(R)$ for which $\theta_s\notin T^\infty_{(\kk(s),\varsigma^s)}$;
\item pairs $(s,k)$ consisting of a point $s\in\spec^\varsigma(R)$ with $\theta_s\in T^\infty_{(\kk(s),\varsigma^s)}$ and a finite field $k$ with $(\bar{k},\varphi_k)$ extending $(\kk(s),\varsigma^s)$ with $(\bar{k},\varphi_k)\not\models\theta_s$.
\end{enumerate}

\end{corollary}
\begin{proof}
In view of  \ref{effdirqe-frob}, the argument of \ref{sent-over-Z} becomes effective. Given a sentence $\theta$ over $(R,\varsigma)$, a primitive recursive procedure computes a $\varsigma$-localisation $R'$ of $R$, an integer $N>0$ and a basic Galois stratification on $S'=\spec(R',\varsigma)$
$$
\cA=\langle S',(Z,\Sigma)/(S',\varsigma),C\rangle.
$$
such that
$$
\theta^\sharp\equiv^{\frob,N}_{(R',\varsigma)}\cA^\sharp.
$$
Then $\theta\in T^{\infty}_{(R,\varsigma),\gen}$ if and only if $C=\Sigma$, which can be verified by a primitive recursive test. 

In the case where we started with $(R,\varsigma)=(\F_q,\varphi_r)$, and we are given a $\theta\in T^\infty_{(\F_q,\varphi_r)}$, potential exceptions must be sought among those finite fields $k$ of size at most $N$ for which $(\bar{k},\varphi_k)$ extends $(\F_q,\varphi_r)$ and we need to check whether $(\bar{k},\varphi_k)$ satisfies $\theta$ or not. 

For each of those $k$, substituting $\varphi_k$ for the difference operator in $\theta$ yields a first-order sentence in the language of rings on an algebraically closed field $\bar{k}$, and such statements can be decided by a well-known primitive recursive procedure for algebraically closed fields.

In the case of $R=\Z[1/n]$, the algorithm produced an explicit $n'$ with $n|n'$ so that $R'=\Z[1/n']$. In order to find the exceptional cases, it suffices to consider only the characteristics dividing $n'$. Thus, for each prime $p$ dividing $n'$, we apply the first part of the corollary to decide whether $\theta_p\in T^\infty_p$. If so, we use the above algorithm for listing the exceptional finite fields $k$ of characteristic $p$ for which $(\bar{k},\varphi_k)\not\models\theta_p$.  
\end{proof}

\begin{corollary}\label{prdec-frob}
The theories $T_0^\infty$ and $T_p^\infty$ are primitive recursive decidable. Moreover,
if $\theta\in T_0^\infty$, a primitive recursive procedure can compute the finite list of:
\begin{enumerate}
\item primes $p$ with $\theta_p\notin T_p^\infty$;
\item pairs $(p,q)$ such that $\theta_p\in T^\infty_p$ but $(\bar{\F}_p,\varphi_q)\not\models\theta_p$.
\end{enumerate} 
\end{corollary}

\subsection{Applications to effective difference algebra}

Our results have direct consequences for effective/constructive difference algebra, allowing us to, for example, rediscover the primitive recursiveness of the \emph{perfect ideal membership problem} (Cohn states that the problem can be solved by an `effective procedure', see 
\cite[Chapter 8, no.~14]{cohn}). Recall, an ideal $I$ in a difference ring is called \emph{perfect}, if $a\sigma(a)\in I$ implies that $a$ and $\sigma(a)$ are both in $I$. The \emph{perfect closure} $\{E\}$ of a set $E$ is the least perfect ideal containing $E$.  
\begin{corollary}\label{pr-perfidmem}
Let $(k,\varsigma)$ be an effective difference field. A primitive recursive procedure can decide, given difference polynomials $f,f_1\ldots,f_n$ over $(k,\varsigma)$, whether
$$
f\in\{f_1,\ldots,f_n\}.
$$ 
\end{corollary}
\begin{proof}
Following \cite[2.19]{ive-tch}, \cite[2.29]{ive-tgsacfa}, the following conditions are equivalent
\begin{enumerate}
\item $f\in\{f_1,\ldots,f_n\}$;
\item $V(f_1,\ldots,f_n)\subseteq V(f)$;
\item $\acfa_{(k,\varsigma)}\vdash \forall x_1\cdots \forall x_m \bigwedge_{i\leq n} f_i(x_1,\ldots,x_m)=0 \rightarrow f(x_1,\ldots,x_m)=0$,
\end{enumerate}
and the last condition is decidable by a primitive recursive procedure via \ref{prdec}.
\end{proof}
\appendix

\section{Directly presented difference schemes}\label{s:comp-dir-pres}

The goal of this Appendix is to show how our framework of direct presentations relates to the notion of \emph{directly presented difference schemes} from \cite{udi}, and assumes that the reader is familiar with basic notation and concepts of that paper. 

\begin{proposition}[{\cite[Sect.~4.3]{udi}}, {\cite[3.1, 3.2]{ive-tgs}}]
Let $(R,\varsigma)$ be a difference ring. 
\begin{enumerate}
\item The forgetful functor
from the category of difference $(R,\varsigma)$-algebras to the category of $R$-algebras
has a left adjoint $[\varsigma]_R$, i.e., for every $R$-algebra $A$ we have a homomorphism
$A\to[\varsigma]_RA$ inducing the functorial isomorphism
$$
\Hom_{(R,\varsigma)}([\varsigma]_RA,(C,\sigma))=\Hom_R(A,C),
$$
for every $(R,\varsigma)$-algebra $(C,\sigma)$.  

\item Let $X$ be an (algebraic) scheme over $R$. The functor
from the category of difference $(R,\varsigma)$-schemes to the category of sets,
$(Z,\sigma)\mapsto\Hom_R(Z,X)$ (morphisms of locally $R$-ringed spaces)
is representable. More precisely, we have a difference scheme $[\varsigma]_RX$ associated to $X$, equipped with the universal morphism
$[\varsigma]_RX\to X$ of locally $R$-ringed spaces, inducing a functorial isomorphism
$$
\Hom_R(Z,X)=\Hom_{(R,\varsigma)}((Z,\sigma),[\sigma]_RX),
$$
for every $(R,\sigma)$-difference scheme $(Z,\sigma)$.
% where the morphisms on the left are the morphisms of locally $R$-ringed spaces.   
\end{enumerate}
\end{proposition}

\begin{definition}\label{dir-pres-def}
Let $(R,\varsigma)$ be a difference ring. An $(R,\varsigma)$-algebra $(A,\sigma)$ is \emph{directly presented} if there exists an $(R,\varsigma)$-epimorphism $(P,\sigma)\to (A,\sigma)$ from some difference polynomial ring 
$(P,\sigma)=[\varsigma]_RR[\bar{x}]=R[\bar{x}]_\sigma$ whose kernel $I$ is $\sigma$-generated by $I\cap R[\bar{x},\sigma(\bar{x})]$. %$I\cap P_1$.
\end{definition}
Intuitively speaking, there is a choice of a tuple of generators $a\in A$ such that $A$ is $\sigma$-generated by $a$ over $R$ and the relations between the generators are all deduced from the relations between $a$ and $\sigma a$.

\begin{definition}[{\cite[6.1]{udi}}]
Let $(R,\varsigma)$ be a difference ring. Let $Y$ be an algebraic scheme over $R$ and let $Z$ be a closed subscheme of $Y\times_{\spec(R)}Y_\varsigma$. Let $\Gamma_\sigma$ denote the graph of $\sigma$ on $[\varsigma]_RY\times_{\spec^\varsigma(R)}[\varsigma]_RY_\varsigma$ and we define the difference scheme $$[\varsigma]_R(Y,Z)$$ as the projection to $[\varsigma]_RY$ of the difference scheme $[\varsigma]_RZ\cap\Gamma_\sigma$.
\end{definition}

\begin{remark}\label{dir-diff-corr}
With the above notation, let $(F,\varphi)$ be a difference field extending $(R,\varsigma)$. Given a point $y\in Y(F)$, clearly $y{\varphi}\in Y_\varsigma(F)$. We have
$$
[\varsigma]_R(Y,Z)(F,\varphi)=\{y\in Y(F):(y,y\varphi)\in Z(F)\},
$$ 
a familiar notion of a difference scheme defined by a correspondence $Z$ between an algebraic scheme $Y$ and its twist $Y_\varsigma$. %already successfully used in \cite{angus-gendif}.
\end{remark}

\begin{lemma} Let $\pi:R[\bar{x}]_\sigma\to(A,\sigma)$ be an epimorphism of difference algebras over a difference ring $(R,\varsigma)$. Let $\bar{a}=\pi(\bar{x})$ be an associated choice of $\sigma$-generators of $A$, and let $X_n$ be the projective system of Zariski closures constructed in \cite[Subsection~2.1]{ive-tgs}. The following conditions are equivalent:
\begin{enumerate}
\item $\pi$ is a direct presentation of $(A,\sigma)$ over $(R,\varsigma)$;
\item $X\simeq[\varsigma]_R(X_0,X_1)$. 
\end{enumerate}
\end{lemma}

\begin{remark}\label{strong-dirpres}
With notation from the previous lemma,  if we have that $X_{n+1}\simeq X_n\times_{X_{n-1,\varsigma}}X_{n,\varsigma}$ for $n\geq 1$, then $X$ is directly presented in a very strong sense.
\end{remark}

The following results illustrate how near an arbitrary difference scheme is to a directly presented one.
\begin{fact}
\begin{enumerate}[wide]%[align=left, leftmargin=*]
\item If $(R,\varsigma)\to(A,\sigma)$ is a morphism of transformal domains of finite $\sigma$-type, then a finite $\sigma$-localisation of $A$ is directly presented. This follows from a known result of difference algebra \cite[Theorem~3.2.6]{wibmer-dif} that a localisation of $A$ is finitely $\sigma$-presented over $A$, followed by a simple choice of a longer tuple of generators in order to get a direct presentation.

\item The `Preparation Lemma' from \cite[Subsection~2.1]{ive-tgs} yields the strong form of direct presentation through \ref{strong-dirpres}, provided we shrink both $A$ and $R$. 

\par
\item An affine or projective difference scheme of finite total dimension over $\Z$ or a difference field can be embedded as a closed subscheme into a directly presented scheme with the same underlying algebraically reduced well-mixed structure, see \cite[Corollary~4.36]{udi}.
\end{enumerate}
\end{fact}

Since directly presented difference schemes will mostly be used in the context where we will be interested only in their points with values in difference fields, we seek a framework that describes them suitably along the lines of \ref{dir-diff-corr}, and which is easily extended to generalised difference schemes.

\begin{remark}\label{dirprespts}
For a direct presentation $(X,\sigma)$, in view of \ref{dir-diff-corr} and \ref{dirpts}, we have that 
$$[\varsigma]_R(X_0,X_1)(F,\varphi)=(X,\sigma)^\sharp(F,\varphi)\simeq(X,\sigma)(F,\varphi).$$
 Thus, in considerations of difference field-valued points, we can neglect the distinction between a direct presentation and its associated difference scheme. %and we may even refer to the objects of $\sadir$ as `(almost) directly presented difference schemes'.
\end{remark}

\begin{remark}\label{dirgal-vs-gal}
Suppose that $(X,\Sigma)/(Y,\sigma)$ is a \emph{finite} Galois cover of transformally integral difference schemes of finite transformal type over a transformal domain $(R,\varsigma)$ with group $G$ as in \cite{ive-tch} (the extension of associated function fields is algebraically finite). By $\sigma$-localising $Y$ (and $X$), we can obtain a direct Galois cover with a rather special property that $\rlexp{()}{\pi_1}:G_1\to G_0$ is an isomorphism (and both groups are isomorphic to $G$).

Note that more general direct covers considered in this paper (in spite of having finite fibres) correspond to situations in which the extension of the underlying function fields is algebraic (of finite transformal type) but not necessarily finite.
\end{remark}

\end{document}